\documentclass[11pt, oneside]{article}   	
\usepackage{geometry}                		
\usepackage{graphicx}				
								
\usepackage[all]{xy}
\CompileMatrices

\usepackage{amssymb}
\usepackage{amsmath}
\usepackage{stmaryrd}
\usepackage{amsthm}
\usepackage{tikz-cd}
\usepackage{extarrows}
\usepackage[mathscr]{euscript}
\usepackage{mathrsfs} 
\usepackage{verbatim}
\usepackage[OT2,T1]{fontenc}
\DeclareSymbolFont{cyrletters}{OT2}{wncyr}{m}{n}
\DeclareMathSymbol{\Sha}{\mathalpha}{cyrletters}{"58}
\DeclareMathSymbol{\Che}{\mathalpha}{cyrletters}{"51}

\usepackage{calligra}
\usepackage{mathrsfs}

\newcommand{\Ga}{{\mathbf{G}}_{\rm{a}}}
\newcommand{\Gm}{{\mathbf{G}}_{\rm{m}}}

\DeclareMathOperator{\Gal}{Gal}

\DeclareMathOperator{\Br}{Br}

\DeclareMathOperator{\Pic}{Pic}

\DeclareMathOperator{\Ext}{Ext}

\DeclareMathOperator{\Hom}{Hom}

\DeclareMathOperator{\R}{R}

\DeclareMathOperator{\im}{im}

\DeclareMathOperator{\et}{\acute{e}t}

\newcommand*{\Z}{\ensuremath{\mathbf{Z}}}                        
\newcommand*{\Q}{\ensuremath{\mathbf{Q}}}                     
\newcommand*{\A}{\ensuremath{\mathbf{A}}}                        
\renewcommand*{\P}{\ensuremath{\mathbf{P}}}                        
\newcommand*{\calO}{\mathcal{O}}                                  
\newcommand*{\address}{Einstein Institute of Mathematics, The Hebrew University of Jerusalem, Edmond J. Safra Campus, 91904, Jerusalem, Israel}
\newcommand*{\email}{zevrosengarten@gmail.com}

\usepackage{rotating}

\usepackage{bm}

\numberwithin{equation}{section}

\newtheorem{theorem}{Theorem}[section]

\newtheorem{lemma}[theorem]{Lemma}
\newtheorem{proposition}[theorem]{Proposition}
\newtheorem{corollary}[theorem]{Corollary}

\theoremstyle{definition}
  \newtheorem{definition}[theorem]{Definition}
  
  \theoremstyle{definition}
  \newtheorem{example}[theorem]{Example}

\theoremstyle{remark}
  \newtheorem{remark}[theorem]{Remark}

\usepackage[OT2,T1]{fontenc}

\tikzset{commutative diagrams/.cd,
mysymbol/.style={start anchor=center,end anchor=center,draw=none}
}

\usepackage{stmaryrd}
\usepackage{hyperref}

\title{\textbf{THE RESTRICTED PICARD FUNCTOR}}
\author{Zev Rosengarten \thanks{
While completing this work, the author was supported by Israel Science Foundation Grant No.\,2083/24.
}}

\date{}
\begin{document}
\maketitle

\begin{abstract}
We prove in significant generality the (almost-)representability of the Picard functor when restricted to smooth test schemes. The novelty lies in the fact that we prove such (almost-)representability beyond the proper setting.
\end{abstract}

\setcounter{tocdepth}{1}
\tableofcontents{}

\section{Introduction}

The Picard functor of a variety plays an important role in number theory and algebraic geometry. Some fundamental examples are the dual variety of an abelian variety and the Jacobian of a projective curve, both of which are defined in terms of Picard schemes. Picard schemes also arise in the construction of the Albanese variety of a proper variety, as well as in many other contexts.

Representability results for Picard functors are therefore of fundamental importance. A wonderful theorem, proved by Murre \cite{murre} and Oort \cite{oort} building on work of Grothendieck \cite{fga}, states that the Picard functor $\Pic_{X/K}$ of a proper scheme over a field $K$ is represented by a scheme \cite[Th.\,2]{murre}. Beyond the proper case one cannot expect such representability, as one already sees in the case of the affine line. Indeed, the Picard group of the affine line over a field is trivial, so were $\Pic_{\A^1/K}$ represented by a scheme (necessarily locally of finite type), it would have to be infinitesimal. But for a reduced scheme $X$ that is not seminormal, one does not generally have $\Pic(\A^1_X) = \Pic(X)$, so the desired representability cannot hold. A better, more conceptual obstacle to representability is that a simple cohomological calculation using the ring of dual numbers as a test scheme shows that the tangent space to $\Pic_{X/K}$ at the identity is given by ${\rm{H}}^1(X, \calO_X)$, which will not be finite-dimensional in general beyond the proper setting.

One observes that the problem in the above examples arises from test schemes that are not ``nice.'' Thus Achet, following a suggestion of Raynaud, was led to introduce the restricted Picard functor -- which is the Picard functor but with one's test schemes restricted to be smooth over $K$ -- and proved its representability for forms of affine space admitting a regular compactification \cite[Th.\,1.1]{achetpic}. (By a {\em form} $X$ of affine $n$-space we mean an fppf form. By the Nullstellensatz, this is the same as saying that $X_L \simeq \A^n_L$ as $L$-schemes for some finite field extension $L/K$.)

\begin{definition}
\label{restrictedpicdef}
Let $K$ be a field and $X$ a $K$-scheme. The {\em naive restricted Picard functor}
\[
(\Pic^+_{X/K})^{\rm{naive}}\colon \{\mbox{smooth } K-\mbox{schemes}\} \rightarrow \{\mbox{abelian groups}\}
\]
is the functor defined by the formula $T \mapsto \Pic(X \times T)$. The {\em restricted Picard functor} $\Pic^+_{X/K}$ is the \'etale sheafification of $(\Pic^+_{X/K})^{\rm{naive}}$.
\end{definition}

\begin{remark}
\label{sheafifyrmk}
The \'etale sheafification in the definition of the restricted Picard functor is essential if one wishes to obtain representability results, as schemes are always sheaves for the \'etale (and even the fpqc) site, by the general theory of descent. The same technical issue arises in the definition of the usual Picard functor; see \cite[\S9.2]{fgaexplained}. For an example of a $K$-scheme $X$ such that the map $\Pic(X) \rightarrow \Pic^+_{X/K}(K)$ is not surjective, see Example \ref{picnotnaive}.
\end{remark}

Unfortunately, even the restricted Picard functor cannot be represented by a scheme in general. To illustrate the problem, let $\overline{C}/K$ be a smooth projective geometrically connected curve of genus $g > 0$, let $x, y \in \overline{C}(K)$ be distinct rational points, and let $C := \overline{C}\backslash\{x, y\}$. Then $\Pic_{\overline{C}/K}$ is an extension of $\Z$ by an abelian variety, the Jacobian $J$ of $\overline{C}$. Because $\Pic(\overline{C} \times T) \twoheadrightarrow \Pic(C \times T)$ for every smooth (or even regular) $K$-scheme $T$, one can conclude that $\Pic^+_{C/K}$ is the quotient of $J$ by the point $j$ represented by the divisor $[x] - [y]$. If $j$ is not a torsion point, then this quotient is not a scheme. Thus we see that an inevitable obstacle to representability comes from quotients of group schemes by subgroups generated by points of infinite order.

We may now state the main result of the present paper. For a commutative, locally finite type $K$-group scheme $G$, we denote by $G^t$ the preimage in $G$ of the torsion subgroup of the \'etale component group $G/G^0$.

\begin{theorem}
\label{almostreprestpic}
If $K$ is a field and $X$ is a regular $K$-scheme of finite type, then there is an isomorphism $\Pic^+_{X/K} \simeq G/E$ of sheaves on the smooth-\'etale site of ${\rm{Spec}}(K)$, where $G$ is a smooth commutative $K$-group scheme with finitely generated component group and $E \subset G^t$ is an \'etale $K$-group scheme such that $E(K_s)$ is a free abelian group of finite rank. Furthermore, $G^0$ is unique up to unique isomorphism.
\end{theorem}

 For results on the structure of the group $G^0$ appearing in Theorem \ref{almostreprestpic}, see \S\ref{structuresection}. As we shall see, the most technically demanding case of Theorem \ref{almostreprestpic} is when ${\rm{char}}(K) = p > 0$. Indeed, when ${\rm{char}}(K) = 0$ one may entirely skip \S\S\ref{boundednesssec}--\ref{descpurelyinpsepsection}. Let us remark, however, that even for algebraically closed $K$ of characteristic $p$, these sections are required, because even though $K$ is perfect, function fields of varieties over $K$ are typically not.
 
 Theorem \ref{almostreprestpic} is nice, but what we would really like are results asserting representability by a scheme. Using the above theorem, we will deduce the following.

\begin{theorem}
\label{represtpic}
Let $K$ be a field and $X$ a regular $K$-scheme of finite type with dense smooth locus. Assume that any map from $X$ into an abelian variety $A$ factors through a finite subscheme of $A$. Then $\Pic^+_{X/K}$ is represented by a smooth commutative $K$-group scheme with finitely generated component group. Furthermore, $(\Pic^+_{X/K})^0$ is unipotent.
\end{theorem}

As examples of the above, let us mention two particular results.

\begin{theorem}
\label{geomunirrep}
Let $K$ be a field and let $X$ be a regular $K$-scheme of finite type with dense smooth locus. If $X_{\overline{K}}$ admits a dense open subscheme that is rationally connected, then $\Pic^+_{X/K}$ is represented by a smooth commutative $K$-group scheme with finitely generated component group. Furthermore, $(\Pic^+_{X/K})^0$ is unipotent.
\end{theorem}

\begin{proof}
Abelian varieties do not admit nonconstant maps from $\P^1$ by \cite[Prop.\,6.1]{milne}, hence also not from rationally connected varieties. So $X$ admits no nonconstant map to an abelian variety over $\overline{K}$, hence also not over $K_s$. Now apply Theorem \ref{represtpic}.
\end{proof}

As an easy consequence of the above, we will prove the following theorem, which generalizes Achet's result to arbitrary forms of affine space, with no assumption on the existence of a regular compactification. Following Achet, we call a smooth connected unipotent $K$-group $U$ {\em strongly wound} when it admits no nontrivial unirational $K$-subgroup. Equivalently, there is no nonconstant $K$-morphism $X \rightarrow U$ with $X$ an open subscheme of $\P^1_K$.

\begin{theorem}
\label{formsofaffrep}
If $K$ is a field and $X$ is a $K$-form of $\A^n_K$, then $\Pic^+_{X/K}$ is represented by a smooth commutative unipotent $K$-group scheme with strongly wound identity component. If $X(K) \neq \emptyset$, then, for every smooth $K$-scheme $S$, $\Pic^+_{X/K}(S) = \Pic(X \times S)/\Pic(S)$.
\end{theorem}

Among the examples falling under the rubric of Theorem \ref{formsofaffrep} are smooth connected unipotent groups. These may indeed have positive-dimensional restricted Picard functors when $K$ is imperfect. For general and specific classes of examples, see, for instance, \cite[Th.\,6.7.9]{kmt} and \cite[Props.\,5.2, 5.4, 5.5, Ex.\,5.8]{rostrans}.

The proof of Theorem \ref{almostreprestpic} proceeds roughly as follows. If $X$ admits a regular compactification $\overline{X}$ whose complementary divisors are all geometrically irreducible, then for any smooth connected $K$-scheme $T$, one has an exact sequence
\[
\oplus_{D} \Z \longrightarrow \Pic(\overline{X} \times T) \longrightarrow \Pic(X \times T) \longrightarrow 0,
\]
where the sum is over the divisors on $\overline{X}$ disjoint from $X$. We know that the proper $\overline{X}$ has representable Picard scheme, and we can use this and the above sequence to construct $\Pic^+_{X/K}$ as a quotient of $\Pic_{\overline{X}/K}$ by a finitely generated subgroup. (This is a slight simplification: Because $\Pic_{\overline{X}/K}$ may be nonsmooth, one actually forms the quotient of a maximal smooth subgroup scheme.) This is roughly speaking the strategy used in \cite{achetpic} to deal with forms of affine space having regular compactification.

The problem, of course, is that we do not know whether $X$ admits a regular compactification. Even when $K$ is algebraically closed, in positive characteristic the existence of a regular compactification remains a formidable open problem. Fortunately, however, there is a weaker form of resolution which is adequate for many purposes: de Jong's alterations theorem \cite[Th.\,4.1]{alterations}. (Actually, we will require a somewhat more precise result from \cite{alterations2}.) The difference between an alteration and a resolution is that we relax the requirement that the resolution be a birational map and only require that it be generically finite. We may replace $X$ by a dense open subscheme and thereby assume that the alteration $X' \rightarrow X$ is finite flat. We then use faithfully flat descent to describe a line bundle on $X$ in terms of a descent datum for a line bundle on $X'$. This is helpful because $X'$ admits a regular compactification and therefore has ``almost-representable'' restricted Picard functor in the sense of Theorem \ref{almostreprestpic}. 

The key difficulty in showing that the functor whose points consist of a line bundle plus descent datum on $X'$ is (almost) representable is to prove that the descent data which arise in this manner are ``uniformly bounded'' in a suitable sense. To illustrate the difficulty, consider the situation in which $X = \A^1_K$ and $X' = \A^1_{K'}$ for some finite extension field $K'/K$. (Of course, this is in a sense a silly example because $\Pic^+_{X/K} = 0$ in this case, as one may see directly. But it illustrates the point most simply -- without some of the technical window dressing -- and so the author hopes that the reader will indulge him and pretend that he does not realize this already.) A line bundle on $X'$ is then automatically trivial, and a descent datum for $X'_T/X_T$ consists of nothing other than a unit on $(X' \times_X X')_T = (\A^1_T)_{K' \otimes_K K'}$ (which is to say, an automorphism of the trivial line bundle) that satisfies the cocycle condition. The problem is that $K' \otimes_K K'$ may well be non-reduced when $K'/K$ is not separable. Therefore a unit on $(\A^1_{K'\otimes_K K'})_T$ may have arbitrarily high degree as a ``polynomial'' in $\Gamma(T_{K'\otimes_K K'}, \calO)[Y]$ on $\A^1$. (We put the word polynomial in scare quotes because when $T$ is not quasi-compact, a global section of $(\A^1_{K'\otimes_K K'})_T$ may have infinite degree.) Without bounding this degree a priori, we would obtain not representability by a scheme, but by an ind-scheme, because the set of all polynomials is not represented by a scheme. Thus we require some sort of bound on the degree. This will follow from the cocycle condition, which says that our unit $u$ satisfies 
\begin{equation}
\label{cocycleintro}
u_{12}u_{23} = u_{13},
\end{equation}
where $u_{ij}$ is the image of $u$ under the ``projection'' map $K'\otimes_K K' \rightarrow K' \otimes_K K' \otimes_K K'$ which sends $a \otimes b$ to $a \otimes b$ in the $(i,j)$ coordinates and $1$ in the other coordinate. Looking at the leading terms in (\ref{cocycleintro}), degree considerations would allow us to conclude that the degree of $u$ is bounded (that $u$ is constant, in fact) if $K' \otimes_K K' \otimes_K K'$ were an integral domain. Unfortunately, this is not the case (the non-reducedness of this ring being the whole cause of our troubles). But what we will show is a poor man's version of the domain property, which says that if $a, b \in K' \otimes_K K'$ satisfy $a_{12}b_{23} = 0$, then either $a$ or $b$ must vanish. (In fact, we prove something somewhat more general; see Proposition \ref{intdomprop}.) This will then allow us to conclude. The argument in general is similar, but with some additional technical complications.

The structure of this paper is as follows. In \S\ref{descGaloissection} we descend the almost-representability of the restricted Picard functor through finite Galois covers. In \S\ref{boundednesssec} we discuss the notion of boundedness for a collection of descent data, and show how boundedness of the collection of all possible descent data on certain line bundles implies the almost-representability of the restricted Picard functor. In \S\ref{descpurelyinpsepsection} we use the notion of boundedness to construct the tools required to descend almost-representability of the restricted Picard functor through finite radicial covers. In \S\ref{almostrepsection} we prove the almost-representability of the restricted Picard functor for regular $K$-schemes of finite type. Finally, in \S\ref{structuresection}, we study the structure of the restricted Picard functor.

\subsection{Acknowledgements} I thank K\k{e}stutis \v{C}esnavi\v{c}ius and the anonymous referees for many helpful comments, suggestions, and errata which have greatly improved the paper.

\subsection{Notation And Conventions}

Throughout this paper, $K$ denotes a field, and $K_s$, $K_{\rm{perf}}$, $\overline{K}$ denote separable, perfect, and algebraic closures of $K$, respectively. When it appears, $p > 0$ is a prime denoting the characteristic of $K$.

\section{Descent through Galois covers}
\label{descGaloissection}

In this section we descend the (almost-)representability of the restricted Picard functor through finite Galois covers of schemes. We first prove the following proposition on the representability of unit groups.

\begin{proposition}
\label{repofunits}
Let $f\colon X \rightarrow {\rm{Spec}}(K)$ be a normal $K$-scheme of finite type. There is a unique subfunctor ${\rm{Units}}_{X/K}$ of the functor $$f_*\Gm\colon \{\mbox{$K$-schemes}\} \rightarrow \{\mbox{abelian groups}\}$$ with the following two properties:
\begin{itemize}
\item[(i)] ${\rm{Units}}_{X/K}$ is represented by a smooth $K$-scheme.
\item[(ii)] ${\rm{Units}}_{X/K}(T) = (f_*\Gm)(T)$ for every smooth $K$-scheme $T$.
\end{itemize}
Furthermore, ${\rm{Units}}_{X/K}$ is a commutative $K$-group scheme such that the $K_s$-points of its component group are free of finite rank, and ${\rm{Units}}_{X/K}^0 = \R_{A/K}(\Gm)$, where $A \subset \Gamma(X, \calO_X)$ is the ring of elements algebraic over $K$.
\end{proposition}

\begin{proof}
The uniqueness of ${\rm{Units}}_{X/K}$ follows from Yoneda's Lemma, as does the fact that it is a commutative $K$-group scheme. To prove existence and the other assertions, we are free by Galois descent to replace $K$ by a finite Galois extension. First assume that $X$ is separated, and choose a normal compactification $\overline{X}$ of $X$. (First use the Nagata compactification theorem \cite[Tag 0F41]{stacks} to obtain a compactification, and then normalize.) Let $A := \Gamma(\overline{X}, \calO_{\overline{X}})$, a finite $K$-algebra. We have an inclusion $\R_{A/K}(\Gm) \hookrightarrow f_*\Gm$. We are free to replace $K$ by a finite Galois extension, so we may assume that all of the codimension-$1$ boundary components $D_1, \dots, D_n$ of $\overline{X}\backslash X$ are geometrically irreducible. For any smooth $K$-scheme $T$, $X_T$ is still normal, and a uniformizer of $D_i$ is also a uniformizer of $(D_i)_T$. We thus obtain a homomorphism $$\phi\colon(f_{\et})_*\Gm \rightarrow \Z^n$$ given by taking valuations along each of the $D_i$, where $f_{\et}$ is the map on \'etale sites induced by $f$. Then $\ker(\phi)(T) = (\overline{f}_*\Gm)(T)$ for every smooth $T/K$, where $\overline{f}\colon \overline{X} \rightarrow {\rm{Spec}}(K)$ is the structural morphism. If $A := \Gamma(\overline{X}, \calO_{\overline{X}})$, a finite $K$-algebra, then $\overline{f}_*\Gm = \R_{A/K}(\Gm)$. Replacing $K$ by a finite Galois extension, we may assume that $\im(\phi)$ is constant and that every $K$-point in $\im(\phi)$ lifts to a point of $(f_{\et})_*\Gm$. So there exist units $u_1, \dots, u_n$ on $X$ whose images freely generate $\im(\phi)$. Then every unit on $X_T$ for smooth connected $T$ is in the free subgroup generated by the $u_i$ times a unit extending over $\overline{X}_T$, which is to say, an element of $\R_{A/K}(\Gm)(T)$. Thus we obtain a smooth $K$-group scheme ${\rm{Units}}_{X/K}$ which is an extension of a free finite rank constant group scheme by $\R_{A/K}(\Gm)$, which computes the units of $X_T$ for smooth $T$, and which is a subfunctor of the functor of units (because $u_i$ and points of $\R_{A/K}(\Gm)$ define units on $X_T$ for any $K$-scheme $T$). This completes the proof when $X$ is separated.

Now consider the general, possibly non-separated case. Let $U \subset X$ be a separated dense open subscheme (for instance, a dense affine open). Replacing $K$ by a finite Galois extension if necessary, we may assume that all of the boundary components of $X\backslash U$ are geometrically irreducible. Let $D_1, \dots, D_m$ be these components. For any smooth scheme $T$, $X_T$ is still normal, and a uniformizer of $D_i$ is also a uniformizer of $(D_i)_T$. We thus obtain a homomorphism of group schemes $$\phi\colon {\rm{Units}}_{U/K} \rightarrow \Z^m$$ given by taking valuations along each of the $D_i$, and $\ker(\phi)(T) = (f_*\Gm)(T)$ for every smooth $T/K$. Because the codomain of $\phi$ is \'etale, $\ker(\phi)^0 = {\rm{Units}}_{U/K}^0$ and $\ker(\phi)$ is smooth. Furthermore, the algebraic subrings of $\Gamma(X, \calO_X)$ and $\Gamma(U, \calO_U)$ agree by normality of $X$. Thus $\ker(\phi)^0 = \R_{A/K}(\Gm)$, hence $\ker(\phi)^0$ defines a subfunctor of $f_*\Gm$. A similar argument to the one earlier involving the group freely generated by the $u_i$ then shows that $\ker(\phi)$ is a subfunctor of $f_*\Gm$, hence ${\rm{Units}}_{X/K} = \ker(\phi)$. That the component group is free of finite rank follows from the analogous assertion for $U$.
\end{proof}

In the process of descending Picard functors through Galois covers, we will have to make use of some cohomology of finite groups, and in fact to compute certain functors defined in terms of such groups. We therefore study such functors below.

\begin{definition}
\label{gpcohdef}
Let $S$ be a scheme, $\mathscr{M}$ a sheaf of abelian groups on the smooth-\'etale site of $S$, and $G$ a group acting on $\mathscr{M}$. Then $\mathscr{H}^i(G, \mathscr{M})$ is the (\'etale) sheafification of the presheaf $T \mapsto {\rm{H}}^i(G, \mathscr{M}(T))$.
\end{definition}

\begin{proposition}
\label{H^i(G,et)}
If $E$ is an \'etale commutative $K$-group scheme, and $G$ is a finite group acting on $E$, then $\mathscr{H}^i(G, E)$ is an \'etale $K$-group scheme. If $E(K_s)$ is finitely generated, then $\mathscr{H}^i(G, E)(K_s)$ is finitely generated for $i = 0$ and finite for $i > 0$.
\end{proposition}

\begin{proof}
By \'etale descent, we may assume that $E$ is the constant group scheme associated to a $G$-module $M$. Then $\mathscr{H}^i(G, E)$ is the constant $K$-group scheme associated to the group ${\rm{H}}^i(G, M)$, which proves the first assertion. For the second, it suffices to note that ${\rm{H}}^i(G, M)$ is finitely generated (for instance, by the cocycle definition of group cohomology), and that, for $i > 0$, it is also torsion, hence finite.
\end{proof}

\begin{proposition}
\label{repcohshftm}
For a commutative $K$-group scheme $M$ of finite type equipped with an action by the finite group $G$, $\mathscr{H}^i(G, M)$ is represented by a smooth commutative $K$-group scheme of finite type, which is affine if $M$ is so.
\end{proposition}

\begin{proof}
First, nothing changes if we replace $M$ by its maximal smooth $K$-subgroup scheme, so we may assume that $M$ is smooth. The functor $\mathcal{Z}^n(G, M)$ of $n$-cocycles $G^n \rightarrow M$ is readily checked to be representable by the closed subscheme of $M^{G^n}$ corresponding to the cocycle condition. We also have coboundary maps $d_{n+1}\colon M^{G^n} \rightarrow \mathcal{Z}^{n+1}(G, M)$. Because $M^{G^n}$ is smooth, the quotient group sheaf $Q^i := \mathcal{Z}^i(G, M)/{\rm{im}}(d_i)$ on the smooth-\'etale site of $K$ agrees with the usual fppf quotient group scheme. (That is, the \'etale quotient is already a scheme by \cite[Tag 0B8G]{stacks}, hence an fppf sheaf.) Replacing $Q^i$ by its maximal smooth $K$-subgroup scheme (\cite[Lem.C.4.1, Rem.\,C.4.2]{cgp}) does not change the underlying functor on the smooth-\'etale site of $K$, so we see that $\mathscr{H}^i(G, M)$ is represented by a smooth commutative finite type $K$-group scheme, which is affine if $M$ is.
\end{proof}

\begin{proposition}
\label{repcohshgen}
Let $M$ be a commutative $K$-group scheme with finitely generated component group, and suppose that $M$ comes equipped with an action by a finite group $G$. Then ${\mathscr{H}}^i(G, M)$ is (represented by) a smooth $K$-group scheme with finitely generated component group. For $i > 0$, ${\mathscr{H}}^i(G, M)$ is finite type, and it is affine if $M^0$ is so.
\end{proposition}

\begin{proof}
As before, we may replace $M$ by its maximal smooth $K$-subgroup scheme and thereby assume that $M$ is smooth. Let $E := M/M^t$ be the torsion-free quotient of the component group of $M$. By Galois descent, we may assume that $K$ is separably closed, so that $E$ is constant and the quotient map $M \rightarrow E$ admits a section, so that $M \simeq M^t \times E$. We are thus reduced to the \'etale and finite type cases, which are handled by Propositions \ref{H^i(G,et)} and \ref{repcohshftm}.
\end{proof}

\begin{remark}
One may of course -- with appropriate modifications -- prove analogues of the above results for cohomology sheaves of group schemes over an arbitrary base. We will never use this, however.
\end{remark}

Let us introduce a definition that describes the type of representability that arises in Theorem \ref{almostreprestpic}.

\begin{definition}
\label{almostrepdef}
We say that a sheaf $\mathscr{F}$ of abelian groups on the smooth-\'etale site of ${\rm{Spec}}(K)$ is {\em almost-representable} when it is of the form $G/E$ for a smooth commutative $K$-group scheme $G$ with finitely generated component group and a subsheaf $E \subset G$ represented by an \'etale $K$-group scheme $E$ with $E(K_s)$ finitely generated.
\end{definition}

We will soon obtain a slight strengthening of the above definition that is sometimes useful (Lemma \ref{quotalmrep}).

\begin{lemma}
\label{exttors}
Let $E, E'$ be \'etale $K$-group schemes with finitely generated groups of $K_s$-points. Then $\Ext^1(E, E')$ is a torsion group.
\end{lemma}

\begin{proof}
Filtering $E$ and $E'$ between their torsion subgroups and the free quotients, we may assume that $E(K_s)$ and $E'(K_s)$ are free of finite rank. Then $\Ext^1_{K_s}(E, E') = 0$, so $\Ext^1(E, E') = {\rm{H}}^1(\Gal(K_s/K), \Hom(E(K_s), E'(K_s)))$. Since higher Galois cohomology is torsion, the lemma follows.
\end{proof}

The following lemma gives a useful strengthening of the definition of almost-representable sheaves.

\begin{lemma}
\label{quotalmrep}
Let $\mathscr{F}$ be an almost-representable sheaf on the smooth-\'etale site of ${\rm{Spec}}(K)$. Then one has an isomorphism $\mathscr{F} \simeq G/E$, where $G$ is a smooth commutative $K$-group scheme with finitely generated component group, and $E \subset G^t$ is a subsheaf represented by an \'etale $K$-group scheme such that $E(K_s)$ is free of finite rank.
\end{lemma}

\begin{proof}
Let $\mathscr{F} \simeq H/F$ be a presentation of $\mathscr{F}$ as an almost-representable sheaf. Let $C := H/H^t$ denote the torsion-free quotient of the component group of $H$, and let $\phi\colon F \rightarrow C$ denote the composition $F \subset H \rightarrow C$. First assume that $\phi(F) = 0$. Then $F \subset H^t$. Let $T \subset F$ denote the torsion subgroup. Because $T$ is finite, $\overline{H} := H/T$ is a smooth $K$-group scheme with finitely-generated component group. (It is a scheme by \cite[Tag 0B8G]{stacks}, and smooth by \cite[IV$_1$, Prop.\,17.3.3]{ega}.) Then $\mathscr{F} \simeq H/F = \overline{H}/(F/T)$ gives a presentation of $\mathscr{F}$ of the desired form.

Now assume that $\phi(F) \subset C$ is nonzero. Then $F$ defines an element of $\Ext^1(\phi(F), \ker(\phi))$. By Lemma \ref{exttors}, this element is killed by some $N > 0$. The multiplication by $N$ map induces an isomorphism $\phi(F) \simeq [N]\phi(F)$, so the extension of group schemes
\[
0 \longrightarrow \ker(\phi) \longrightarrow \phi^{-1}([N]\phi(F)) \xlongrightarrow{\phi} [N]\phi(F) \longrightarrow 0
\]
splits. Thus there is a $K$-subgroup scheme $F'$ of $F$ mapping isomorphically onto $[N]\phi(F)$. Then $H/F'$ is an extension of $C/[N]\phi(F)$ by $H^t$, hence (\cite[Tag 0B8G]{stacks} again) representable by a scheme $G'$ with finitely generated component group. Letting $\overline{F} := F/F'$, we have $H/F \simeq G'/\overline{F}$, and $\overline{F}(K_s)$ has strictly smaller rank than $F(K_s)$, hence $G'/\overline{F}$ has the form asserted in the lemma by induction.
\end{proof}

\begin{lemma}
\label{repettors}
Suppose given a short exact sequence of sheaves on the smooth-\'etale site of ${\rm{Spec}}(K)$
\[
1 \longrightarrow H \longrightarrow \mathscr{G} \xlongrightarrow{\pi} Q \longrightarrow 1
\]
with $H$ and $Q$ locally finite type group schemes over $K$. Then $\mathscr{G}$ is represented by a smooth $K$-group scheme, and $\pi$ is a smooth morphism.
\end{lemma}

\begin{proof}
We first show that $\mathscr{G}$ is a scheme. By \cite[Tag 0B8G]{stacks}, it suffices to show that $\mathscr{G}$ is a separated algebraic space over $K$. First, replacing $H$ and $Q$ by their maximal smooth $K$-subgroup schemes has no effect on the resulting sheaves, so we may assume that both groups are smooth. Choose an \'etale surjection $f\colon X \rightarrow Q$ such that $f$ lifts to a point $s \in \mathscr{G}(X)$. Then $Y := \mathscr{G} \times_Q X \simeq H \times X$ via $(g, x) \mapsto (gs(x)^{-1}, x)$, so this gives a (locally finite type) $K$-scheme admitting an \'etale surjection onto $\mathscr{G}$. It remains to prove that $\mathscr{G}$ is separated.

We must show that the diagonal map $\Delta\colon \mathscr{G} \hookrightarrow \mathscr{G} \times \mathscr{G}$ is a closed embedding. This may be done after base change to some \'etale cover of the target. We take as our cover the scheme $Y \times Y$. Thus we wish to verify that the top row of the Cartesian diagram
\[
\begin{tikzcd}
Y \times_{\mathscr{G}} Y \arrow[r, hookrightarrow] \arrow{d} \arrow[dr, phantom, "\square"] & Y \times Y \arrow{d} \\
\mathscr{G} \arrow{r}{\Delta} & \mathscr{G} \times \mathscr{G}
\end{tikzcd}
\]
is a closed embedding. Now $Y \times_{\mathscr{G}} Y$ is the subscheme of $Y \times Y = (H \times X)^2$ defined by the condition $h_1s(x_1) = h_2s(x_2)$. This condition in particular implies that $f(x_1) = f(x_2)$, so $Y \times_{\mathscr{G}} Y$ is contained in the preimage $Z \subset Y \times Y$ of the diagonal under the map $Y \times Y \rightarrow Q$. Because $Q$ -- being a locally finite type group scheme over a field -- is separated, $Z$ is a closed subscheme of $Y \times Y$. Finally, $Y \times_{\mathscr{G}} Y$ is the preimage of $1 \in H(K)$ under the map $Z \rightarrow H$, $(h_1, x_1, h_2, x_2) \mapsto h_1s(x_1)(h_2s(x_2))^{-1}$, hence it is a closed subscheme of $Z$.

Replacing a representing object for $\mathscr{G}$ by its maximal smooth $K$-subgroup, we may assume that $\mathscr{G}$ is a smooth $K$-group, and we wish to show that $\pi$ is smooth. This may be checked after an \'etale base change, hence it is equivalent to check that $Y \simeq H \times X \rightarrow X$ is smooth, and this is clear.
\end{proof}

We now establish certain permanence properties for almost-representable sheaves.

\begin{lemma}
\label{keralrepmap}
The kernel of a homomorphism between almost-representable sheaves is almost-representable.
\end{lemma}

\begin{proof}
Let $\phi\colon \mathscr{F} \rightarrow \mathscr{G}$ be the map. Write $\mathscr{F} \simeq F/E_F$ and $\mathscr{G} = G/E_G$, where $F, G$ are smooth commutative $K$-group schemes with finitely-generated component groups, and $E_F\subset F$, $E_G \subset G$ are subsheaves represented by \'etale $K$-group schemes with finitely generated groups of $K_s$-points. Let $q_F\colon F \rightarrow \mathscr{F}$ denote the quotient map and similarly for $q_G\colon G \rightarrow \mathscr{G}$. Let $W$ be the sheaf defined by the following Cartesian diagram:
\[
\begin{tikzcd}
W \arrow{rr}{\eta} \arrow{d}{\zeta} \arrow[drr, phantom, "\square"] && G \arrow{d}{q_G} \\
F \arrow{r}{q_F} & \mathscr{F} \arrow{r}{\phi} & \mathscr{G}
\end{tikzcd}
\]
The map $q_G$ is surjective with kernel isomorphic to $E_G$, hence the same holds for $\zeta$. Lemma \ref{repettors} therefore implies that $W$ is a smooth $F$-group scheme. Furthermore, the subsheaf $\eta^{-1}(E_G) \subset W$ surjects onto $\ker(\phi)$ via $q_F\circ \zeta$. Call this surjection $\delta$. Then $\ker(\delta) = \zeta^{-1}(E_F)$. This is an extension of $E_F$ by $E_G$, hence \'etale with finitely generated component group. In order to prove the lemma, therefore, it is enough to show that $U := \eta^{-1}(E_G) \subset W$ is a smooth $K$-group scheme with finitely generated component group.

The sheaf $U$ is defined by the following Cartesian diagram:
\[
\begin{tikzcd}
U \arrow{r} \arrow{d} \arrow[dr, phantom, "\square"] & W \arrow{d}{\eta} \\
E_G \arrow[r, hookrightarrow] & G
\end{tikzcd}
\]
It follows that $U$ is computed by the locally finite type $K$-group scheme $S$ defined by the same Cartesian diagram in the category of schemes, and replacing $S$ by its maximal smooth $K$-subgroup scheme, we see that $U$ is a smooth $K$-group scheme. It remains to prove that it has finitely generated component group. Because $U$ is a closed subgroup subscheme of $S$, it is enough to prove the same for $S$. 

Because $W$ is an $E_G$-torsor over $F$, it has finitely generated component group, hence the same holds for the closed $K$-subgroup scheme $\ker(\eta)_{\rm{gp}} \subset W$, defined as the kernel in the category of $K$-group schemes. (Note: This is not generally the same as the kernel in the category of sheaves on the smooth-\'etale site of $K$. The relationships is that the latter is the maximal smooth $K$-subgroup scheme of the former.) But $\ker(\eta)_{\rm{gp}} = \ker(S \rightarrow E_G)_{\rm{gp}}$. It follows that $S$ has finitely generated component group.
\end{proof}

\begin{lemma}
\label{Ginvalrep}
Let $\mathscr{F}$ be a sheaf on the smooth-\'etale site of $K$, equipped with an action by a finite group $G$. If $\mathscr{F}$ is almost-representable, then so is the sheaf $\mathscr{F}^G$ of $G$-invariants.
\end{lemma}

\begin{proof}
Apply Lemma \ref{keralrepmap} to the map $\mathscr{F} \rightarrow \prod_{g \in G}\mathscr{F}$ defined by $f \mapsto ((g-1)f)_g$.
\end{proof}

\begin{lemma}
\label{extalrepaff}
A commutative extension of an almost-representable group sheaf by a finite type $K$-group scheme with finitely generated component group is almost-representable.
\end{lemma}

\begin{proof}
Suppose given an extension $\mathscr{F}$ of $G/E$ by $H$ with $H$ an affine $K$-group scheme of finite type, $G$ a smooth commutative $K$-group scheme with finitely generated component group, and $E \subset G$ an \'etale $K$-group scheme with $E(K_s)$ free of finite rank. Replacing $H$ by its maximal smooth $K$-subgroup scheme, we may as well assume that $H$ is smooth. Consider the following commutative diagram with exact rows and columns, in which $W$ is defined by the indicated Cartesian diagram:
\[
\begin{tikzcd}
&& E \arrow[d, hookrightarrow] \arrow[r, equals] & E \arrow[d, hookrightarrow] & \\
0 \arrow{r} & H \arrow[d, equals] \arrow{r} & W \arrow[d, twoheadrightarrow] \arrow{r} \arrow[dr, phantom, "\square"] & G \arrow[d, twoheadrightarrow] \arrow{r} & 0 \\
0 \arrow{r} & H \arrow{r} & \mathscr{F} \arrow{r} & G/E \arrow{r} & 0
\end{tikzcd}
\]
By Lemma \ref{repettors}, $W$ is a (smooth) $G$-scheme. Furthermore, the middle row shows that $W$ has finitely generated component group. It follows that $\mathscr{F} \simeq W/E$ is almost-representable.
\end{proof}

We may now show that almost-representability of the restricted Picard functor descends through finite Galois covers.

\begin{proposition}
\label{alrepdescgal}
Let $Y \rightarrow X$ be a finite Galois cover of normal $K$-schemes of finite type. If $\Pic^+_{Y/K}$ is almost-representable, then so is $\Pic^+_{X/K}$.
\end{proposition}

\begin{proof}
Let $G$ be the Galois group. Then $Y_T \rightarrow X_T$ is a $G$-cover for any $K$-scheme $T$, and we have the Hochschild-Serre spectral sequence, functorial in $T$,
\[
E_2^{i,j} = {\rm{H}}^i(G, {\rm{H}}^j(Y_T, \Gm)) \Longrightarrow {\rm{H}}^{i+j}(X_T, \Gm).
\]
By sheafifying the low-degree exact sequence associated to this spectral sequence, we obtain an exact sequence of sheaves
\[
0 \longrightarrow \mathscr{H}^1(G, {\rm{Units}}_{Y/K}) \longrightarrow \Pic^+_{X/K} \longrightarrow (\Pic^+_{Y/K})^G \xlongrightarrow{f} \mathscr{H}^2(G, {\rm{Units}}_{Y/K}).
\]
Propositions \ref{repofunits} and \ref{repcohshgen} imply that ${\rm{H}}^i(G, {\rm{Units}}_{Y/K})$ is a smooth affine $K$-scheme for $i > 0$. Combining Lemmas \ref{Ginvalrep} and \ref{keralrepmap}, therefore, we deduce that $\ker(f)$ is almost-representable. Lemma \ref{extalrepaff} then implies that $\Pic^+_{X/K}$ is almost-representable. 
\end{proof}

\section{Boundedness}
\label{boundednesssec}

We will ultimately prove the almost-representability of $\Pic^+_{X/K}$ by relating line bundles on $X$ to line bundles on a suitable fppf cover $Y$ of $X$ together with descent data. A descent datum for a line bundle amounts to a trivialization of a certain associated line bundle. We will have to understand those trivializations that show up as descent data, and in particular show that they yield a representable functor. The key to accomplishing this is the notion of {\em boundedness} for a family of trivializations.

\begin{definition}
\label{boundedness}(Boundedness)
Let $Z$ be a qcqs $K$-scheme, and let $V \subset \Gamma(Z, \calO_Z)$ be a finite-dimensional $K$-vector space. Suppose given a $K$-scheme $T$ and $s \in \Gamma(Z \times_K T, \calO_{Z \times_K T})$. We say that the pair $(T, s)$ is {\em $V$-bounded} when $s \in V \otimes_K \Gamma(T, \calO_T)$.

Let $S$ be a $K$-scheme, and let $\mathscr{L}, \mathscr{M}$ be line bundles on $Z \times_K S$, and $\phi\colon \mathscr{L} \xrightarrow{\sim} \mathscr{M}$ an isomorphism. Suppose given a pair $(g, \psi)$ with $g\colon T \rightarrow S$ an $S$-scheme and $\psi\colon ({\rm{Id}}_Z, g)^*(\mathscr{L}) \xrightarrow{\sim} ({\rm{Id}}_Z, g)^*(\mathscr{M})$ an isomorphism. We say that $(g, \psi)$ is {\em $(\phi, V)$-bounded} when the pair $$(T, \psi\circ({\rm{Id}}_Z, g)^*(\phi)^{-1}) \in \Gamma(Z\times_K T, \Gm))$$ and its inverse $(T, ({\rm{Id}}_Z, g)^*(\phi)\circ \psi^{-1})$ are both $V$-bounded.

Fix $\phi$. We say that a family $F := \{(g_{\alpha}, \psi_{\alpha})\}_{\alpha \in A}$ of pairs as above is {\em bounded} when there exists a finite-dimensional $V$ such that every element of the family is $(\phi, V)$-bounded. Note that this is independent of $\phi$ if $S$ is quasi-compact. In particular, we say (without reference to any $\phi$) that $F$ is {\em locally bounded} when there is a Zariski cover $\{U_i\}_{i \in I}$ of $S$ such that, for each $i$, $U_i$ is quasi-compact and $F_i := \{(g_{\alpha,i}, \psi_{\alpha,i})\}$ is bounded, where $g_{\alpha,i} := g_{\alpha}|g_{\alpha}^{-1}(U_i)$ and $\psi_{\alpha,i} = \psi_{\alpha}|{Z \times_K g_{\alpha}^{-1}(U_i)}$. Note that we require the existence of an isomorphism $\phi$ to make this definition, but the definition is independent of which isomorphism we choose.
\end{definition}

\begin{proposition}
\label{bdedrepble}
Fix $Z, S, \mathscr{L}, \mathscr{M}, \phi, V$ as in Definition $\ref{boundedness}$ with $S$ a finite type $K$-scheme. Then the set-valued functor on the category of $K$-schemes sending $T$ to the set of all $(\phi, V)$-bounded pairs $(g, \psi)$ with $g\colon T \rightarrow S$ a $K$-morphism and $\psi\colon ({\rm{Id}}_Z, g)^*(\mathscr{L}) \xrightarrow{\sim} ({\rm{Id}}_Z, g)^*(\mathscr{M})$ an isomorphism, is represented by a finite type $K$-scheme ${\rm{Isom}}_V$. Furthermore, for $V \subset W \subset \Gamma(Z, \calO_Z)$ finite-dimensional $K$-subspaces, the natural inclusion ${\rm{Isom}}_V \subset {\rm{Isom}}_W$ is a closed immersion.
\end{proposition}

\begin{proof}
We first prove the representability. When $Z = \emptyset$, the assertion is clear, so assume $Z \neq \emptyset$. Once one has specified $g$, to give a $(\phi, V)$-bounded isomorphism $\psi$ is the same as giving a unit $u \in \Gamma(Z \times_K T, \Gm)$ such that $u, u^{-1} \in V \otimes_K \Gamma(T, \calO_T)$. The datum of an element $f \in V \otimes_K \Gamma(T, \calO_T)$ is the same as a $K$-morphism from $T$ to the vector group $\underline{V}$ defined by $V$. Thus to specify $\psi$ is the same as specifying a pair $(f, f')$ of $K$-morphisms $T \rightarrow \underline{V}$ such that $ff' = 1$. We claim that this condition defines a closed subscheme $Y \subset \underline{V} \times \underline{V}$, in which case we see that ${\rm{Isom}}_V = S \times Y$ is a scheme.

To see that $ff' = 1$ indeed defines a closed subscheme of $\underline{V} \times \underline{V}$, choose a basis $\mathscr{B} := \{v_1, \dots, v_n\}$ of $V$, with $1 \in \mathscr{B}$ if $1 \in V$. Let $V' \subset \Gamma(Z, \calO_Z)$ be a finite-dimensional subspace with $V\cdot V \subset V'$ and $1 \in V'$. Choose a basis $\mathscr{B}' := \{1, b_2, \dots, b_m\}$ of $V'$. For each $1 \leq i, j \leq n$, write 
\begin{equation}
\label{bdedrepblepfeqn1}
v_iv_j = \sum_{1 \leq k \leq m} \lambda_{ijk}b_k
\end{equation}
for unique $\lambda_{ijk} \in K$. Then $\underline{V} \simeq \A^n$ via the coefficients of the $v_i$, and if we let $d_i, e_i$ denote the coordinates on $\A^n$, then the condition $ff' = 1$ is the same as expanding the left side in terms of the basis $\mathscr{B}'$ using the equations (\ref{bdedrepblepfeqn1}), and then equating the coefficients of the $b_k$ on both sides. This yields a collection of equations defining a closed subscheme of $\underline{V}^2$. This completes the proof that ${\rm{Isom}}_V$ is representable.

To prove that ${\rm{Isom}}_V \hookrightarrow {\rm{Isom}}_W$ is a closed immersion, we note that it suffices to show that the map $\underline{V}^2 \hookrightarrow \underline{W}^2$ is a closed immersion, and this is well-known (and easy): Just complete the chosen basis $\mathscr{B}$ for $V$ to a basis $\mathscr{B} \sqcup \mathscr{C}$ for $W$, and then $\underline{V}$ is defined by the conditions that the coefficients of the elements of $\mathscr{C}$ vanish.
\end{proof}

Because we will analyze line bundles on a scheme $X$ by considering descent data on some fppf cover $f\colon Y \rightarrow X$, the following definition will play an important role in proving the representability of $\Pic^+_{X/K}$.

\begin{definition}
\label{desccandef}
Let $f\colon Y \rightarrow X$ be a morphism of $K$-schemes. We define the presheaf of abelian groups ${\rm{DesCan}}_f^{\rm{naive}}$ on the smooth-\'etale site of $K$ by the formula $$T \mapsto \ker\left((\pi_1)_T^* - (\pi_2)_T^*\colon \Pic(Y \times_K T) \rightarrow \Pic(Y \times_X Y \times_K T)\right),$$ where $\pi_i\colon Y \times_X Y \rightarrow Y$ are the projections. We define ${\rm{DesCan}}_f$ to be the (\'etale) sheafification of ${\rm{DesCan}}_f^{\rm{naive}}$. We also denote these functors by ${\rm{DesCan}}_{Y/X}^{\rm{naive}}$ and ${\rm{DesCan}}_{Y/X}$, respectively.
\end{definition}

The importance of boundedness arises from the following proposition.

\begin{proposition}
\label{bddcocycimprep}
Suppose given an fpqc morphism of connected normal finite type $K$-schemes $f\colon Y \rightarrow X$. Assume that ${\rm{DesCan}}_{Y/X}$ is almost-representable, and choose an isomorphism ${\rm{DesCan}}_{Y/X} \xrightarrow{\sim} G/E$ exhibiting it as such. Choose an \'etale cover $U \rightarrow G$ such that the element of ${\rm{DesCan}}_{Y/X}(G)$ corresponding to the quotient map $G \rightarrow G/E$ is represented on $U$ by a line bundle $$\mathscr{L} \in \ker(\pi_1^* - \pi_2^*\colon \Pic(Y \times_K U) \rightarrow \Pic(Y \times_X Y \times_K U)).$$ Let $F$ be the family of pairs $(u, \psi)$ where $u \in U(K_s)$ and $\psi\colon (\pi_1)_{K_s}^*(u^*\mathscr{L}) \xrightarrow{\sim} (\pi_2)_{K_s}^*(u^*\mathscr{L})$ is an isomorphism satisfying the cocycle condition
\[
\pi_{12}^*(\psi)\circ \pi_{23}^*(\psi) = \pi_{13}^*(\psi),
\]
where $\pi_{ij}\colon Y \times_X Y \times_X Y \rightarrow Y \times_X Y$ is projection onto the $(i,j)$ factors. If $F$ is locally bounded, then $\Pic^+_{X/K}$ is almost-representable.
\end{proposition}

\begin{proof}
Let $P$ be the presheaf on the smooth-\'etale site of ${\rm{Spec}}(K)$ defined by the following Cartesian diagram:
\[
\begin{tikzcd}
P \arrow{r}{h'} \arrow{d} \arrow[dr, phantom, "\square"] & U \arrow{d} \\
(\Pic^+_{X/K})^{\rm{naive}} \arrow{r}{f^*} & {\rm{DesCan}}_{Y/X}^{\rm{naive}}
\end{tikzcd}
\]
so that the (\'etale) sheafification $\mathscr{F}$ of $P$ lives in the following Cartesian diagram:
\[
\begin{tikzcd}
\mathscr{F} \arrow{r}{h} \arrow{d} \arrow[dr, phantom, "\square"] & U \arrow{d}{a} \\
\Pic^+_{X/K} \arrow{r}{b} & {\rm{DesCan}}_{Y/X}
\end{tikzcd}
\]
We will first show that there is a locally finite type, geometrically reduced $K$-algebraic space computing $\mathscr{F}$. In order to prove this, we choose an open cover $\{U_i\}_{i \in I}$ of $U$ by finite type $K$-schemes. Then, for each $i$, the family of $(u, \psi)$ as in the proposition such that $u \in U_i(K_s)$ is bounded.

Let $P_i$ denote the preimage of $U_i$ in $P$. For a smooth $K$-scheme $T$, to give a $T$-point of $P_i$ is the same as giving a morphism $g\colon T \rightarrow U_i$ together with a line bundle (up to isomorphism) $\mathscr{M}$ on $X \times_K T$ such that $({\rm{Id}}_Y, g)^*(\mathscr{L}) = f_T^*(\mathscr{M})$. By the theory of faithfully flat descent, this is (functorially) the same as the information of the morphism $g$ together with a descent datum (up to isomorphism) on $({\rm{Id}}_Y, g)^*(\mathscr{L})$ with respect to the map $f_T$. Fix an isomorphism $\phi\colon (\pi_1)_U^*(\mathscr{L}) \xrightarrow{\sim} (\pi_2)_U^*(\mathscr{L})$. Then a descent datum over $T$ is described by a unit $u \in \Gamma(Y \times_X Y \times_K T, \Gm)$ such that $u\cdot g^*(\phi)$ satisfies the cocycle condition. By assumption, the set of such $u$, and the set of $u^{-1}$ for such $u$, is bounded when we restrict to the case $T = {\rm{Spec}}(K_s)$ -- that is, it is contained in $V \otimes_K \Gamma(T, \calO_T)$ for some finite-dimensional $V \subset \Gamma(Y \times_X Y, \calO_{Y \times_X Y})$. For a general $T$, locally on $T$ one has that $u, u^{-1} \in W \subset \Gamma(Y \times_X Y, \calO_{Y \times_X Y})$ for some finite-dimensional $W \supset V$ (which $W$ is allowed to vary with $T$ and with the open subset of $T$ over which one is working locally). Thus the descent datum for $({\rm{Id}}_Y, g)^*(\mathscr{L})$ lies in ${\rm{Isom}}_W$ in the notation of Proposition \ref{bdedrepble}, while the restriction to each $K_s$-point of $T$ factors through the closed subscheme ${\rm{Isom}}_V$. Because $T$ is smooth, $T(K_s)$ is dense in $T$, so it follows that all of $T$ maps into ${\rm{Isom}}_V$. So we see that every descent datum for $\mathscr{L}$ arises from a morphism $T \rightarrow {\rm{Isom}}_V \times U$, and we claim that the condition for the descent datum to define a cocycle defines a closed subscheme of ${\rm{Isom}}_V \times U$. Indeed, the cocycle condition on $u\cdot g^*(\phi)$ is described by the condition
\begin{equation}
\label{bddcocycimpreppfeqn1}
\pi_{12}^*(u)\pi_{23}^*(u) = g^*(C)\pi_{13}^*(u),
\end{equation}
where $C := \pi_{13}^*(\phi)\pi_{23}^*(\phi^{-1})\pi_{12}^*(\phi^{-1})$. Choose a finite-dimensional subspace $H \subset \Gamma(Y \times_XY \times_X Y, \calO)$ containing $C\cdot \pi_{13}^*(V)$ and $\pi_{12}^*(V)\cdot \pi_{23}^*(V)$, and choose a basis $\mathscr{B}$ for $H$. Then (\ref{bddcocycimpreppfeqn1}) reduces to the equations equating the coefficients of each element of $\mathscr{B}$ on both sides of (\ref{bddcocycimpreppfeqn1}). This condition defines a closed subscheme of ${\rm{Isom}}_V \times U$, as claimed.

We have thus obtained for each finite-dimensional $W \subset \Gamma(Y\times_X Y, \calO)$ a finite type $U$-scheme $r\colon S'_W \rightarrow U$ which is a subfunctor of the functor of all descent data for $\mathscr{L}|Y \times_K U$, in such a manner that the functor of all descent data is the filtered direct limit of the $S'_W$, and such that, for each $i$, there is a finite-dimensional $V_i$ such that, whenever $V_i \subset W$, $S'_{W,i} := S'_{W}|r^{-1}(U_i)$ defines all descent data on $\mathscr{L}|Y \times_K U_i$ over any smooth $K$-scheme $T$. Replacing $S'_{W,i}$ by its maximal geometrically reduced closed subscheme $S_{W,i}$ (see \cite[Lem.C.4.1]{cgp}) does not alter its $T$-points for smooth $T$. Since $S'_{W_1,i} \subset S'_{W_2,i}$ is a closed subscheme with the same $K_s$-points when $V_i \subset W_1 \subset W_2$, we see that $S_{W,i}$ is independent of $W \subset V_i$. We denote this common geometrically reduced finite type $U_i$-scheme by $S_i$. The $S_i$ glue together to yield a finite type $U$-scheme $S$ that is geometrically reduced over $K$ and which computes the functor of descent data for $\mathscr{L}$ from $Y$ to $X$ over any smooth $K$-scheme $T$ with a map $T \rightarrow U_i$.

Now distinct descent data for the line bundle $\mathscr{L}$ may be isomorphic. This happens precisely when they differ by a unit on $Y \times_X Y \times_K T$ of the form $\pi_1^*(y)/\pi_2^*(y)$ for some unit $y$ on $Y_T$. We thus obtain an action of ${\rm{Units}}_{Y/K}$ on $\varinjlim_WS'_{W,i}$, where the limit is over finite-dimensional $W \supset V_i$. (See Proposition \ref{repofunits}; note that it is important that ${\rm{Units}}_{Y/K}$ be a subfunctor of the full functor of units on $Y$ in order to obtain this action, since we do not know if $S'_{W,i}$ is smooth.) Furthermore, by descent theory, the units $y$ on $Y$ such that $\pi_1^*(y)/\pi_2^*(y) = 1$ are exactly those coming from $X$ via $f^*$. Thus the action factors through an action of ${\rm{Units}}_{Y/K}/f^*({\rm{Units}}_{X/K})$. We claim that this action is then free. The action of the (functor of) units on $Y$ on $\varinjlim S'_{W,i}$ is free modulo the units on $X$, so the key point that must be checked is that the intersection of the functor of units on $X$ with ${\rm{Units}}_{Y/K}$ is exactly ${\rm{Units}}_{X/K}$. This is the content of Lemma \ref{intunitsYX} below. We thus obtain a free action of ${\rm{Units}}_{Y/K}/{\rm{Units}}_{X/K}$ on $\varinjlim S'_{W,i}$. Because ${\rm{Units}}_{Y/K}$ is smooth, it preserves the maximal geometrically reduced closed subscheme $S_i$, so we obtain a free action of ${\rm{Units}}_{Y/K}/{\rm{Units}}_{X/K}$ on $S$ whose presheaf quotient computes $P$, hence whose \'etale sheafified quotient computes $\mathscr{F}$.

We claim that $H := {\rm{Units}}_{Y/K}/{\rm{Units}}_{X/K}$ is a smooth $K$-group scheme. To verify this, we claim that it suffices to prove that the map on component groups induced by the inclusion $${\rm{Units}}_{X/K} \hookrightarrow {\rm{Units}}_{Y/K}$$ is injective. Indeed, letting $E_X, E_Y$ denote the components groups of ${\rm{Units}}_{X/K}$ and ${\rm{Units}}_{Y/K}$, respectively, this injectivity would yield by the snake lemma an exact sequence
\[
0 \longrightarrow {\rm{Units}}_{Y/K}^0/{\rm{Units}}_{X/K}^0 \longrightarrow {\rm{Units}}_{Y/K}/{\rm{Units}}_{X/K} \longrightarrow E_Y/E_X \longrightarrow 0,
\]
and representability by a scheme then follows as usual from \cite[Tag 0B8G]{stacks}.

To prove the claimed injectivity on component groups, we may pass to a finite Galois extension and thereby suppose that we have a unit $u$ on $X$ that does not live in ${\rm{Units}}_{X/K}^0$, and we must show that its pullback to $Y$ does not live in ${\rm{Units}}_{Y/K}^0$. By Proposition \ref{repofunits}, $u$ is in the identity component of $X$ precisely when it is algebraic over $K$, and similarly for units on $Y$. Thus what we must show is that, if $u$ is not algebraic, then neither is its pullback to $Y$. But this follows from the fact that the pullback map $f^*\colon \Gamma(X, \calO_X) \rightarrow \Gamma(Y, \calO_Y)$ is injective, because $f$ is faithfully flat.
 
It follows that $\mathscr{F}$ is computed by the geometrically reduced finite type $U$-algebraic space $\mathscr{S} := S/H$ with structure map $r'\colon \mathscr{S} \rightarrow U$. Furthermore, $\mathscr{F}_i := h^{-1}(U_i)$ is computed by the finite type $U_i$-algebraic space $\mathscr{S}'_{W,i} := S'_{W,i}/H$ for any finite-dimensional $V \subset W_i$. Let $R \subset U \times U$ be the \'etale equivalence relation such that $U/R = G$. Because $G$ is a locally finite type $K$-group scheme, it is separated, so $R \subset U \times U$ is a closed subscheme. Then $\iota\colon R' := (r' \times r')^{-1}(R) \hookrightarrow \mathscr{S} \times \mathscr{S}$ is an e\'tale equivalence relation and a closed embedding, and we let $\mathscr{X} := \mathscr{S}/R'$, a locally finite type, separated (because $\iota$ is a closed embedding), geometrically reduced $K$-algebraic space that computes the sheaf $\mathscr{G}$ on the smooth-\'etale site of ${\rm{Spec}}(K)$ which sits in the following Cartesian diagram
\[
\begin{tikzcd}
\mathscr{G} \arrow{r} \arrow{d} \arrow[dr, phantom, "\square"] & G \arrow{d} \\
\Pic^+_{X/K} \arrow{r}{b} & {\rm{DesCan}}_{Y/X}
\end{tikzcd}
\]
Note also that, because $\mathscr{S}$ is of finite type over $U$, $\mathscr{X}$ is of finite type over $G$. Similarly, we define for each finite-dimensional $W$ a quotient $\mathscr{X}'_{W,i}$ of $\mathscr{S}'_{W,i}$ that is an algebraic space whose maximal geometrically reduced closed sub-algebraic space is $\mathscr{X}|b^{-1}(a(U_i))$. 

Let $D_i := a(U_i) \subset {\rm{DesCan}}_{Y/X}$. The group structure on $U/R = G$, together with that on (isomorphism classes of) line bundles on $X$ descending $\mathscr{L}$, yields a multiplication map for any $i, j, k$ and $W_1, W_2$ $$\mathscr{X}'_{W_1,i} \times \mathscr{X}'_{W_2,j} \rightarrow \mathscr{X}'_{W_3, {\bf k}},$$ where ${\bf k} \subset I$ is a finite subset such that the union of the $D_k$ contains $D_i\cdot D_j$, and $W_3$ is sufficiently large depending on $W_1, W_2$ (and $\mathscr{X}_{W_3, {\bf k}}$ has the obvious meaning). In this manner one obtains a group structure on the colimit (over all $W, i$) of the $\mathscr{X}'_{W,i}$, hence on the colimit of the maximal geometrically reduced closed subspaces, which is $\mathscr{X}$. Thus $\mathscr{X}$ is a geometrically reduced, separated, locally finite type $K$-group algebraic space, so it is a smooth commutative $K$-group scheme \cite[Tag 0B8G]{stacks}. Because the map $m\colon \mathscr{X} \rightarrow G$ is of finite type, $\mathscr{X}$ has finitely generated component group. Thus, $\Pic^+_{X/K} \simeq \mathscr{X}/E$ is almost-representable, where the inclusion of $E$ in $\mathscr{X}$ is via the inclusion $E \hookrightarrow \mathscr{G}$, and the fact that the restriction of $\mathscr{X}$ to the smooth-\'etale site of ${\rm{Spec}}(K)$ (of which $E$ is an object) computes $\mathscr{G}$.
\end{proof}

\begin{lemma}
\label{intunitsYX}
The intersection of the functor of units on $X$ with ${\rm{Units}}_{Y/K}$ is ${\rm{Units}}_{X/K}$.
\end{lemma}

\begin{proof}
We are free to extend scalars and thereby assume that the component group $E$ of ${\rm{Units}}_{Y/K}$ is constant and that every component arises from a unit on $Y$ (that is, the map ${\rm{Units}}_{Y/K}(K) \rightarrow E(K)$ is surjective), and similarly for $X$. Let $B \subset \Gamma(Y, \calO_Y)$ be the subring of elements algebraic over $K$, so ${\rm{Units}}^0_{Y/K} = \R_{B/K}(\Gm)$, and similarly let $A$ be the ring of elements of $\Gamma(X, \calO_X)$ algebraic over $K$, so ${\rm{Units}}^0_{X/K} = \R_{A/K}(\Gm)$. Let $T$ be a $K$-scheme, and suppose given a unit $u \in {\rm{Units}}_{Y/K}(T)$ such that $u$ is the pullback of a unit on $X_T$. We must show that $u$ comes from an element of ${\rm{Units}}_{X/K}(T)$.

Working locally, we may assume that $T = {\rm{Spec}}(R)$ and that $u$ lies in a single component $C$ of ${\rm{Units}}_{Y/K}$. Choose a unit $v$ on $Y$ defining a point in $C(K)$. Then $C(R) = v(R \otimes_K B)^{\times}$, so $u$ lies in the $R$-module $v(R \otimes_K B)$. By descent theory, an element of $v(R \otimes_K B)$ is the pullback of a section of $X_R$ precisely when its pullbacks along the two projections $(Y \times_X Y)_R \rightarrow Y_R$ agree. Because $R$ is $K$-flat, the elements of $v(R \otimes_K B)$ whose pullbacks agree are $R \otimes_K M$, where $M \subset vB$ are the elements of $vB$ whose pullbacks agree -- that is, that descend to sections of $\Gamma(X, \calO_X)$.

If $M = 0$, then $u = 0$, a contradiction, so $M \neq 0$. Let $0 \neq m \in M$. Because $Y$ is connected and reduced, $B$ is a field, so $m$ is a unit of $Y$, hence $m$ is a unit of $X$. Replacing $v$ by $m$, therefore, we may assume that $v$ descends to a unit of $X$, hence lives in ${\rm{Units}}_{X/K}$. So modifying by $v$, we have reduced to the case in which $C = {\rm{Units}}_{Y/K}^0 = \R_{B/K}(\Gm)$. Then $M = A$, so an element of $(R \otimes_K B)^{\times}$ which is pulled back from $X$ lies in $(R \otimes_K A)^{\times} = {\rm{Units}}_{X/K}^0(R)$.
\end{proof}

\section{Descent through radicial covers}
\label{descpurelyinpsepsection}

In this section we will prove the results required in order to descend the almost-representability of the restricted Picard functor through radicial covers (although we postpone the proof that representability descends through such covers until the next section). This may be regarded as the technical heart of the entire paper. We will proceed by verifying the hypotheses of Proposition \ref{bddcocycimprep}. We begin by identifying ${\rm{DesCan}}$ in this setting.

\begin{proposition}
\label{descanL/K}
Let $X$ be an irreducible affine $K$-scheme, and let $f\colon Y \rightarrow X$ be an affine radicial map. Then ${\rm{DesCan}}_{Y/X} \simeq \Pic^+_{Y/K}$.
\end{proposition}

\begin{proof}
It suffices to show that one has an isomorphism
\[
{\rm{DesCan}}_{Y/X}^{\rm{naive}} \simeq (\Pic^+_{Y/K})^{\rm{naive}}
\]
when one restricts both sides to the category of {\em affine} smooth $K$-schemes. This amounts to showing that, for any smooth affine $T/K$, and every $\mathscr{L} \in \Pic(Y_T)$, one has $\pi_1^*(\mathscr{L}) \simeq \pi_2^*(\mathscr{L})$, where $\pi_i$ are the base changes to $T$ of the projection maps $Y \times_X Y \rightarrow Y$. By \cite[Lem.\,2.2.9]{rostateduality} applied with $G = \Gm$ and $i = 1$, the embedding $((Y \times_X Y)_T)_{\rm{red}} \hookrightarrow (Y \times_X Y)_T$ induces an isomorphism on Picard groups. It therefore suffices to prove that the two maps
\begin{equation}
\label{descanL/Kpfeqn1}
\begin{tikzcd}
((Y \times_X Y)_T)_{\rm{red}} \arrow{r} & (Y \times_X Y)_T \ar[r,shift left=.75ex,"\pi_1"] \ar[r,shift right=.75ex,swap,"\pi_2"] & Y_T
\end{tikzcd}
\end{equation}
agree. Call these two maps $g_1, g_2$. Let $j\colon \eta \rightarrow ((Y \times_K Y)_T)_{\rm{red}}$ be the inclusion of one of the generic points of $((Y \times_K Y)_T)_{\rm{red}}$. It suffices to show that the maps $g_i\circ j$ agree. But the $g_i\circ j$ agree after postcomposing with $f_T\colon Y_T \rightarrow X_T$. Since $f_T$ is radicial, it follows that the $g_i\circ j$ agree.
\end{proof}

Next we prove the local boundedness of the collection of descent data associated to a radicial extension. The following lemma -- akin to an integral domain property for tensor powers of rings -- is simple, but lies at the heart of the proof of Theorem \ref{almostreprestpic}, and in particular, of the boundedness of the collection of descent data for line bundles on a scheme required in Proposition \ref{bddcocycimprep} in order to descend representability.

\begin{proposition}
\label{intdomprop}
Let $A$ be a ring, let $B$ be a finite free $A$-algebra, and let $\pi_{12},\pi_{23}\colon B \otimes_A B \rightarrow B \otimes_A B \otimes_A B$ denote the maps $b \otimes b' \mapsto b\otimes b' \otimes 1$, $1 \otimes b\otimes b'$, respectively. Let $I, I', I''$ be ideals of $A$ such that, for $b, b' \in B$, if $bb' \in I''B$, then either $b \in IB$ or $b' \in I'B$. Then for $\alpha, \beta \in B \otimes_A B$, if $\pi_{12}(\alpha)\pi_{23}(\beta) \in I''(B \otimes_A B \otimes_A B)$, then either $\alpha \in I(B \otimes_A B)$ or $\beta \in I'(B \otimes_A B)$.
\end{proposition}

\begin{proof}
Write $B = \oplus_{i=1}^n Ab_i$ for some $b_1, \dots, b_n \in B$, so $$B \otimes_A B \otimes_A B = \bigoplus_{i=1}^n (B \otimes_A B \otimes b_i),$$ where $\otimes b_i$ means $\otimes_A Ab_i$. Write $\beta = \sum_{i=1}^n (\gamma_i \otimes b_i)$ with $\gamma_i \in B$. We then have $\pi_{23}(\beta) = \sum_{i=1}^n (1 \otimes \gamma_i \otimes b_i)$, so $\pi_{12}(\alpha)\pi_{23}(\beta) = \sum_{i=1}^n [\alpha(1 \otimes \gamma_i) \otimes b_i]$. Because $B \otimes_A B \otimes_A B$ is a free left $B \otimes_A B$-module over the $1 \otimes 1 \otimes b_i$, to prove the lemma we must show that, if $\alpha \in B \otimes B$ is not in the ideal generated by $I$ and $\gamma \in B$ is not in $I'B$, then $\alpha(1 \otimes \gamma) \notin I''(B \otimes_A B)$. Because multiplication by $1 \otimes B$ makes $B \otimes_A B$ into a free $B$-module, this is the same as asserting that, for $b, b' \in B$, $bb' \in I''B \Longrightarrow b \in IB$ or $b' \in I'B$, and this is true by assumption.
\end{proof}

The following somewhat technical lemma will be used in conjunction with Proposition \ref{intdomprop} to prove the boundedness of the descent data arising from descent through a radicial cover.

\begin{lemma}
\label{prodnondiv}
Let $A$ be a Japanese DVR with maximal ideal $\mathfrak{m}$, and let $B$ be a finite free local $A$-algebra which is an integral domain. Assume that $B^m := \{b^m\mid b \in B\} \subset A$ for some $m > 0$. Then there exists $r > 0$ such that, for $b, b' \in B$, if $b, b' \notin \mathfrak{m}B$, then $bb' \notin \mathfrak{m}^rB$.
\end{lemma}

\begin{proof}
Let $\widetilde{B}$ be the normalization of $B$, which is also the normalization of $A$ inside ${\rm{Frac}}(B)$. Because $A$ is Japanese, $\widetilde{B}$ is a finite $A$-module, hence also finite over $B$. It follows that there exists $0 \neq a \in B$ such that $a\widetilde{B} \subset B$. Replacing $a$ by $a^m$, we may assume that $a \in A$. We claim that $\widetilde{B}$ is a DVR. Indeed, it is a Dedekind domain. Suppose that it admits two distinct maximal ideals $\mathfrak{n}_1, \mathfrak{n}_2$, and let $x \in \mathfrak{n}_1 - \mathfrak{n}_2$. Replacing $x$ by $x^m$, we may assume that $x \in A$. But each maximal ideal of $\widetilde{B}$ lies over $\mathfrak{m}$, so $x$ cannot lie in one maximal ideal but not another. Thus $\widetilde{B}$ admits a unique maximal ideal, hence is a DVR.

Let $v := {\rm{ord}}_\mathfrak{m}(a)$. We claim that, for $b \in B$, if $b \notin \mathfrak{m}B$, then $b \notin \mathfrak{m}^{v+1}\widetilde{B}$. To see this, let $\pi \in A$ be a uniformizer. If $b \in \mathfrak{m}^{v+1}\widetilde{B}$, then $b/\pi^{v+1} \in \widetilde{B}$, hence $ab/\pi^{v+1} \in B$, or $ab \in \pi^{v+1}B$. Because $B$ is a free $A$-module, this implies that $b \in \pi B$, proving the claim.

Now we claim that we may take $r = 2(v+1)$ in the lemma. Indeed, if $b, b' \in B$ satisfy $bb' \in \mathfrak{m}^{2(v+1)}B \subset \mathfrak{m}^{2(v+1)}\widetilde{B}$, then -- using the fact that $A \rightarrow \widetilde{B}$ is a local homomorphism of DVR's -- we conclude that either $b$ or $b'$ lies in $\mathfrak{m}^{v+1}\widetilde{B}$; say $b$ does without loss of generality. Then by the claim of the previous paragraph, we deduce that $b \in \mathfrak{m}B$.
\end{proof}

We may now verify the local boundedness of the descent data associated to a radicial cover. For the notion of local boundedness of a family of isomorphisms, see Definition \ref{boundedness}.

\begin{proposition}
\label{locbddesdata}
Let $f\colon Y \rightarrow X$ be a finite flat radicial morphism between connected normal $K$-schemes of finite type. Let $T$ be a smooth $K$-scheme, and $\mathscr{L}$ a line bundle on $Y_T$ such that $\pi_1^*(\mathscr{L}) \simeq \pi_2^*(\mathscr{L}) \in \Pic((Y \times_X Y)_T)$. Then the collection of all descent data for $\mathscr{L}_S$ with respect to the fppf cover $Y_S \rightarrow X_S$ is locally bounded as $S$ varies over all $T$-schemes smooth over $K$.
\end{proposition}

\begin{proof}
We may assume that $Y \neq \emptyset$. The function field extension $K(Y)/K(X)$ is purely inseparable, so when ${\rm{char}}(K) = 0$ it is an isomorphism. Since $f$ is finite and $X$ and $Y$ are normal, it follows that $f$ is an isomorphism, so the proposition is immediate in this case. So from now on we assume that ${\rm{char}}(K) = p > 0$. We may extend scalars and assume that $K$ is separably closed. Further, we may work locally on $T$ and thereby assume that $T$ is quasi-compact, and in this case we must prove boundedness of the collection of descent data. Because the $K$-points of $S$ are Zariski dense in $S$ for any $K$-smooth $S$, we may assume that $S = {\rm{Spec}}(K)$. To prove boundedness of the descent data, recall that a descent datum is an isomorphism between the pullbacks of $\mathscr{L}$ along the two projections $Y \times_X Y \rightrightarrows Y$ which additionally satisfies the cocycle condition. Fix an isomorphism $\phi\colon \pi_1^*(\mathscr{L}) \xrightarrow{\sim} \pi_2^*(\mathscr{L})$. Let $s\colon {\rm{Spec}}(K) \rightarrow T$ be the $T$-scheme structure on $S = {\rm{Spec}}(K)$. Let $\pi_{ij}\colon Y \times_X Y \times_X Y \rightarrow Y \times_X Y$ denote the various projections. Then the cocycle condition on $u\cdot s^*(\phi)$ with $u \in \Gamma((Y \times_X Y)_T, \Gm)$ is equivalent to the equation
\[
\pi_{12}^*(u)\pi_{23}^*(u) = C\pi_{13}^*(u),
\]
where $C := s^*(C')$ with $C' := \pi_{13}^*(\phi)\pi_{23}^*(\phi)^{-1}\pi_{12}^*(\phi)^{-1} \in \Gamma((Y \times_X Y \times_{X} Y)\times_K T, \Gm)$. Because $T$ is quasi-compact, there is a finite-dimensional $W \subset \Gamma(Y \times_X Y \times_X Y, \calO)$ such that $C', (C')^{-1} \in W_T$. Now we forget some of the particulars of our underlying situation, and we simply prove the following: For any finite-dimensional $W \subset \Gamma(Y \times_X Y \times_X Y, \calO)$, there is a finite-dimensional $V \subset \Gamma(Y \times_{X} Y, \calO)$ such that any $u \in \Gamma(Y \times_{X} Y, \calO)$ satisfying
\begin{equation}
\label{locbddesdatapfeqn3}
\pi_{12}^*(u)\pi_{23}^*(u) \in\pi_{13}^*(u)W
\end{equation}
lies in $V$. Then we see that all of the pairs $(s, u)$ as above defining a cocycle satisfy $u, u^{-1} \in V$.

To prove the claim, we first choose a normal compactification $\overline{X}$ of $X$, and let $\overline{Y}$ be the normal closure of $\overline{X}$ in the function field of $Y$. Then $\overline{Y}$ is a compactification of $Y$ living in a commutative diagram
\[
\begin{tikzcd}
Y \arrow[r, hookrightarrow] \arrow{d}{f} & \overline{Y} \arrow{d}{\overline{f}} \\
X \arrow[r, hookrightarrow] & \overline{X}
\end{tikzcd}
\]
We claim that $\overline{f}$ is radicial, and that the open locus $U \subset \overline{X}$ of points above which $\overline{f}$ is flat is a dense open with ${\rm{codim}}_{\overline{X}}(\overline{X}-U) \geq 2$. First, to see that $\overline{f}$ is radicial, let $V := {\rm{Spec}}(R) \subset \overline{X}$ be affine, and let ${\rm{Spec}}(S) = \overline{f}^{-1}(V)$. Let $t > 0$ such that $K(Y)^{p^t} \subset K(X)$. Then for any $s \in S$, $s^{p^t} \in K(X) = {\rm{Frac}}(R)$ is integral over $R$. Because $\overline{X}$ is normal, therefore, $s^{p^t} \in R$. Thus,
\begin{equation}
\label{locbddesdatapfeqn10}
\overline{f}_*\calO_{\overline{Y}}^{p^t} \subset \calO_{\overline{X}}.
\end{equation}
It follows that ${\rm{Spec}}(S) \rightarrow {\rm{Spec}}(R)$ is radicial. To prove the flatness claim, it suffices to prove that $\overline{f}$ is flat above each codimension $1$ point of $x \in \overline{X}$. Since $\overline{X}$ is normal, the local ring of $x$ is a DVR. Thus we seek to prove that the local ring of $\overline{Y}$ at any point above $x$ is $\calO_{\overline{X},x}$-torsion free, and this follows from the fact that $\overline{f}$ is dominant and $\overline{Y}$ is integral.

Let $D_1, \dots, D_n$ be the boundary divisors of $\overline{X}$ (that is, those irreducible divisors not intersecting $X$), and let $D := \sum_{i=1}^n D_i$. For each $N > 0$, consider the coherent sheaf $\calO_{\overline{X}}(ND)$. Let $g\colon \overline{Y} \times_{\overline{X}} \overline{Y} \rightarrow \overline{X}$ be the natural map, and let $V := g^{-1}(U) \subset \overline{Y} \times_{\overline{X}} \overline{Y}$. Note that $Y \times_X Y \subset V$ because $Y$ is flat over $X$. Consider the open embeddings $j\colon Y \times_X Y \hookrightarrow \overline{Y} \times_{\overline{X}} \overline{Y}$ and $j_V\colon V \hookrightarrow \overline{Y} \times_{\overline{X}} \overline{Y}$. Because the restriction of $g$ to $V \rightarrow U$ is flat by our choice of $U$, we have $${\rm{Ass}}_V(j_V^*g^*\calO_{\overline{X}}(ND)) \subset \overline{g}^{-1}({\rm{Ass}}_U(\calO_{\overline{X}}(ND))) = \eta,$$ where $\eta$ is the unique point above the generic point of $U$. By \cite[IV$_2$, Cor.\,5.11.4]{ega}, therefore, the sheaf $$\mathscr{H}_N := (j_V)_*j_V^*g^*\calO_{\overline{X}}(ND)$$ on $\overline{Y} \times_{\overline{X}} \overline{Y}$ is coherent. Because $Y \times_X Y \subset V$, we have $\mathscr{H}_N|_{Y \times_XY} \simeq g^*\calO(ND)|_{Y \times_X Y} \simeq \calO_{Y\times_XY}$, the last isomorphism because $Y$ lies over $X$, and $\calO(ND)|_X \simeq \calO_X$. We thus have an isomorphism $j^*\mathscr{H}_N \xrightarrow{\sim} \calO_{Y \times_X Y}$ which is adjoint to a map $\mathscr{H}_N \rightarrow j_*\calO_{Y\times_XY}$. Let $\mathscr{G}_N$ be the image of this map. It is a coherent sheaf, hence ${\rm{H}}^0(\overline{Y}\times_{\overline{X}}\overline{Y}, \mathscr{G}_N)$ is a finite-dimensional $K$-vector space. We will show that, for a suitable $N$, every $u$ satisfying (\ref{locbddesdatapfeqn3}) lies in ${\rm{H}}^0(\overline{Y}\times_{\overline{X}} \overline{Y}, \mathscr{G}_N) \subset {\rm{H}}^0(\overline{Y}\times_{\overline{X}}\overline{Y}, j_*\calO) = {\rm{H}}^0(Y\times_XY, \calO)$.

Now we claim that membership in $\mathscr{G}_N$ is determined by codimension-one points on the boundary of $\overline{Y} \times_{\overline{X}} \overline{Y}$. More precisely, we have the following lemma.

\begin{lemma}
\label{cod1pts}
Let $Z$ be the disjoint union of the spectra of the local rings of $\overline{Y}\times_{\overline{X}}\overline{Y}$ at the generic points of $g^{-1}(D_1), \dots, g^{-1}(D_n)$, and let $h\colon Z \hookrightarrow \overline{Y}\times_{\overline{X}}\overline{Y}$ be the natural map. Let $\mathscr{C}$ be the cokernel of the map $\mathscr{G}_N \rightarrow j_*\calO_{Y \times_X Y}$. Then the natural map $\mathscr{C} \rightarrow h_*h^*\mathscr{C}$ is an inclusion.
\end{lemma}

\begin{proof}
Consider the following diagram, in which the horizontal arrows are the inclusions and the vertical arrows are induced by $g$:
\[
\begin{tikzcd}
Y\times_X Y \arrow[bend left]{rr}{j} \arrow{d}{g_X} \arrow{r}{j_{Y\times_XY,V}} & V \arrow{d}{g_U} \arrow{r}{j_V} & \overline{Y}\times_{\overline{X}}\overline{Y} \\
X \arrow{r}{j_{XU}} & U &
\end{tikzcd}
\]
The isomorphism $j_{XU}^*\calO_{U}(ND) \xrightarrow{\sim} \calO_X$ is adjoint to an injective map $\calO_U(ND) \rightarrow (j_{XU})_*\calO_X$. Let $\mathscr{A}$ be the cokernel of this map, so we have an exact sequence of quasi-coherent sheaves
\begin{equation}
\label{cod1ptspfeqn2}
0 \longrightarrow \calO_U(ND) \longrightarrow (j_{XU})_*\calO_X \longrightarrow \mathscr{A} \longrightarrow 0.
\end{equation}
Let $R$ be the disjoint union of the spectra of the local rings of $\overline{X}$ at the generic points of $D_1, \dots, D_n$, and let $r\colon R \hookrightarrow U$ be the inclusion. (Note that $R$ does indeed map into $U$ because ${\rm{codim}}_{\overline{X}}(\overline{X}-U) \geq 2$.) Because $U \subset \overline{X}$ is normal, the canonical map
\begin{equation}
\label{cod1ptspfeqn3}
\mathscr{A} \rightarrow r_*r^*\mathscr{A}
\end{equation}
is an inclusion. Let $h'\colon Z \rightarrow V$ be the inclusion (so $h = j_V\circ h'$). Consider the following commutative diagram, which is Cartesian because $g$ is radicial:
\[
\begin{tikzcd}
Z \arrow{d}{t} \arrow{r}{h'} \arrow[dr, phantom, "\square"] & V \arrow{d}{g_U} \\
R \arrow{r}{r} & U
\end{tikzcd}
\]
Because $g_U$ is flat and $r$ is qcqs, the base change map $g_U^*r_*r^*\mathscr{A} \rightarrow h'_*t^*r^*\mathscr{A} = h'_*(h')^*g_U^*\mathscr{A}$ is an isomorphism. Thus, applying $g_U^*$ to the inclusion (\ref{cod1ptspfeqn3}) yields that the canonical adjunction map
\[
g_U^*\mathscr{A} \hookrightarrow h'_*(h')^*g_U^*\mathscr{A}
\]
is an inclusion. Then applying $(j_V)_*$ and using the equality $(j_V)_*(h')_*(h')^* = h_*h^*(j_V)_*$, we obtain that the adjunction
\begin{equation}
\label{cod1ptspfeqn4}
(j_V)_*g_U^*\mathscr{A} \hookrightarrow h_*h^*(j_V)_*g_U^*\mathscr{A}
\end{equation}
is an inclusion.

Since $g_U$ is flat, applying $g_U^*$ to (\ref{cod1ptspfeqn2}) yields an exact sequence
\begin{equation}
\label{cod1ptspfeqn1}
0 \longrightarrow g_U^*\calO_U(ND) \longrightarrow g_U^*(j_{XU})_*\calO_X \longrightarrow g_U^*\mathscr{A} \longrightarrow 0.
\end{equation}
Because $g_U$ is flat and $j_{XU}$ is qcqs, we have an equality $g_U^*(j_{XU})_*\calO_X \simeq (j_{Y\times_XY, V})_*g_X^*\calO_X \simeq (j_{Y\times_XY, V})_*\calO_{Y\times_XY}$. Substituting this into (\ref{cod1ptspfeqn1}) and applying $(j_V)_*$ then yields the exact sequence
\[
\mathscr{G}_N \longrightarrow j_*\calO_{Y\times_X Y} \longrightarrow (j_V)_*g_U^*\mathscr{A}.
\]
We therefore see that $\mathscr{C}$ is a subsheaf of $(j_V)_*g_U^*\mathscr{A}$, so the lemma follows from the inclusion (\ref{cod1ptspfeqn4}).
\end{proof}

For any section s of $\calO_{\overline{X}}$, one has that $s \in \calO_{\overline{X}}(nD)$ for some $n > 0$ depending on $s$. It follows that any section $u$ of $j_*\calO_{Y \times_X Y}$ lies in $\mathscr{H}_n$ for some $n > 0$. We must show that there is a uniform bound on the $n$ required as $u$ varies over global sections satisfying (\ref{locbddesdatapfeqn3}). Because $W$ is finite-dimensional, one similarly concludes that $W|_{\epsilon^{-1}(U)} \subset \epsilon^*(\calO_{\overline{X}}(mD))$ for some $m > 0$, where $\epsilon\colon \overline{Y}\times_{\overline{X}} \overline{Y} \times_{\overline{X}} \overline{Y} \rightarrow \overline{X}$ is the natural map.
 
Let $z$ be one the finitely many codimension-$1$ boundary points of $\overline{Y} \times_{\overline{X}} \overline{Y}$. We will show that there is $N > 0$ depending on $z$ such that every $u$ satisfying (\ref{locbddesdatapfeqn3}) lies in $\mathscr{G}_N$ near $z$. Since there are only finitely many $z$, this will suffice. Let $x := g(z)$ be the image of $z$ in $\overline{X}$, and let $y \in \overline{Y}$ be the (unique) boundary point lying above $x$. Also let $B := \calO_{\overline{Y}, y}$, $A := \calO_{\overline{X}, x}$. Then $A$ is a Japanese -- excellent even -- DVR, and $B$ is finite flat over $A$. Furthermore, by \ref{locbddesdatapfeqn10}, $B^{p^t} \subset A$. We are therefore in the situation of Lemma \ref{prodnondiv}. We claim that we may take $N = m+r-1$, where $r$ is as in that lemma.

To see this, let $\varpi$ be a uniformizer of $A$. Let $u$ satisfy (\ref{locbddesdatapfeqn3}), and let $n$ be the minimal nonnegative integer such that $u \in \mathscr{G}_n$ near $z$. We may assume that $n > 0$ (as otherwise of course $u \in \mathscr{G}_{m+r-1}$ near $z$). We abuse notation and also denote by $u$ its image in $(B \otimes_A B)[1/\varpi]$. Because $u \in \mathscr{G}_n$, one has $\varpi^nu \in B\otimes_A B$. We claim that $\varpi^nu \notin \varpi(B \otimes_A B)$. To see this, note first that $B\otimes_A B$ -- as a flat $A$-module -- is $\varpi$-torsion free. So if $\varpi^nu \in \varpi(B \otimes_A B)$, then one would have $\varpi^{n-1}u \in B \otimes_A B$, hence $u \in \mathscr{G}_{n-1}$ near $z$, in violation of the minimality of $n$. Now suppose for the sake of contradiction that $n \geq m+r$, and multiply both sides of (\ref{locbddesdatapfeqn3}) by $\varpi^{2n}$. Because $W \subset \mathscr{G}_m$, we conclude that
\[
\pi_{12}^*(\varpi^nu)\pi_{23}^*(\varpi^nu) \in \varpi^{n-m}(B \otimes_A B \otimes_A B).
\]
Because $n-m \geq r$, this violates Proposition \ref{intdomprop}, applied with $I = I' = (\varpi)$ and $I'' = (\varpi^r)$. We conclude, therefore, that $n \leq m+r-1$, as claimed. This completes the proof of the proposition.
\end{proof}

\section{Almost-representability of the restricted Picard functor}
\label{almostrepsection}

In this section we prove the almost-representability of the restricted Picard functor for regular $K$-schemes of finite type.

\begin{lemma}
\label{subet=et}
Any subsheaf of an \'etale sheaf $E$ on the smooth-\'etale site of $K$ is also \'etale.
\end{lemma}

\begin{proof}
\cite[Lem.\,5.5]{rosmodulispaces} says that subsheaves of \'etale sheaves on the geometrically reduced \'etale site of ${\rm{Spec}}(K)$ are still \'etale, and the proof of that lemma also goes through for the smooth-\'etale site.
\end{proof}

As a d\'evissage step, we show that, in proving Theorem \ref{almostreprestpic}, we are free to restrict to an open subscheme.

\begin{lemma}
\label{opensubdev}
Let $K$ be a field, $X$ a finite type regular $K$-scheme and $U \subset X$ a dense open subscheme. Then $\Pic^+_{X/K}$ is almost-representable if and only if $\Pic^+_{U/K}$ is.
\end{lemma}

\begin{proof}
Let $j\colon U \rightarrow X$ denote the inclusion, and consider the restriction map $j^*\colon \Pic^+_{X/K} \rightarrow \Pic^+_{U/K}$. For every smooth $K$-scheme $T$, $X \times T$ is regular, hence $\Pic(X \times T) \rightarrow \Pic(U \times T)$ is surjective. It follows that $j^*$ is a surjection of sheaves on the smooth-\'etale site of ${\rm{Spec}}(K)$. We claim that $\ker(j^*)$ is \'etale with finitely generated group of $K_s$-points, which will pprove the only if direction of the lemma. By Galois descent, it suffices to show this after extending scalars to some finite Galois extension $L/K$, so we may assume that all of the codimension-$1$ components of $X\backslash U$ are geometrically irreducible. In that case, if the components are $D_1, \dots, D_n$, then the natural map $\psi\colon \oplus_D \Z \rightarrow \ker(j^*)$ is surjective. By Lemma \ref{subet=et}, $\ker(\psi)$ is a constant group scheme, so $\ker(j^*)$ is \'etale (with finitely generated group of $K_s$-points).

Now we turn to the proof of the if direction (so $K$ need no longer be separably closed). Write $\Pic^+_{U/K} \simeq G/E$ with $G$ a smooth commutative $K$-group scheme with finitely generated component group, and $E \subset G$ \'etale with $E(K_s)$ of finite rank, and let $H$ be defined by the following Cartesian square:
\[
\begin{tikzcd}
H \arrow{r} \arrow{d} \arrow[dr, phantom, "\square"] & \Pic^+_{X/K} \arrow{d}{j^*} \\
G \arrow{r} & \Pic^+_{U/K}
\end{tikzcd}
\]
Then $H$ fits into an exact sequence
\[
0 \longrightarrow \ker(j^*) \longrightarrow H \longrightarrow G \longrightarrow 0.
\]
Lemma \ref{repettors} implies that $H$ is a (smooth) scheme, hence a smooth commutative $K$-group scheme. Its component group is finitely generated, because the same holds for $G$ and $\ker(j^*)$. We thus have that $\Pic^+_{X/K} \simeq H/E$ is almost-representable.
\end{proof}

We now prove the almost-representability of the restricted Picard functor.

\begin{theorem}
\label{represtpicmainbody}
For any regular $K$-scheme $X$ of finite type, $\Pic^+_{X/K}$ is almost-representable.
\end{theorem}

\begin{proof}
Working component by component, we may assume that $X$ is irreducible. Then by \cite[Cor.\,5.15]{alterations2}, there are alterations $Y_2 \xrightarrow{f_2} Y_1 \xrightarrow{f_1} (X_{K_{{\rm{perf}}}})_{\rm{red}}$ with $f_2$ generically Galois and $f_1$ generically purely inseparable (that is, the function field extension $k(Y_1)/k((X_{K_{{\rm{perf}}}})_{\rm{red}})$ is purely inseparable) and such that $Y_2$ admits a regular, hence $K_{{\rm{perf}}}$-smooth, compactification $\overline{Y}_2$. The data descends from $K_{\rm{perf}}$ down to some finite purely inseparable extension $L/K$, and then renaming and postcomposing with the map $(X_L)_{\rm{red}} \rightarrow X$, we find that there exist a finite purely inseparable extension $L/K$ and dominant generically finite morphisms $Y_2 \xrightarrow{f_2} Y_1 \xrightarrow{f_1} X$ between integral finite type $K$-schemes such that (1) $Y_2$ admits an $L$-smooth compactification $\overline{Y}_2$; (2) $f_2$ is generically Galois; and (3) $f_1$ is generically purely inseparable. Shrinking $X$ using lemma \ref{opensubdev}, we may assume that $f_2$ is finite Galois and that $f_1$ is finite flat radicial.

The Picard functor $\Pic_{\overline{Y}_2/L}$ of the proper $L$-scheme $\overline{Y}_2$ is represented by a locally finite type $L$-scheme, which by the classical theorem of the base has finitely generated component group. Furthermore, this functor is the \'etale (rather than merely fppf) sheafification of the naive Picard functor $T \mapsto \Pic(\overline{Y}_2 \times_L T)$ on the category of $L$-schemes, because $\overline{Y}_2$ is geometrically reduced. (Work on each geometric component of $\overline{Y}_2$, and use \cite[Th.\,9.2.5, Exercise 9.3.11]{fgaexplained}.) When we restrict to the category of smooth $L$-schemes, $\Pic_{\overline{Y}_2/L}$ agrees with the functor represented by its maximal smooth $L$-subgroup scheme. It follows that $\Pic^+_{\overline{Y}_2/L}$ is almost-representable (representable even). By Lemma \ref{opensubdev}, so too is $\Pic^+_{Y_2/L}$. Proposition \ref{alrepdescgal} then implies that $\Pic^+_{Y_1/L}$ is almost-representable. Now, just by definition, $$(\Pic^+_{Y_1/K})^{\rm{naive}} = \R_{L/K}((\Pic^+_{Y_1/L})^{\rm{naive}}).$$ Because $L/K$ is finite radicial surjective, it induces an equivalence of \'etale sites \cite[Exp.\,IX, Th.\,4.10]{sga}, hence $\R_{L/K}$ commutes with \'etale sheafification, so we obtain an isomorphism $$\Pic^+_{Y_1/K} \simeq \R_{L/K}(\Pic^+_{Y_1/L}).$$ Since pushforward through the finite morphism ${\rm{Spec}}(L) \rightarrow {\rm{Spec}}(K)$ is an exact functor between \'etale sites, it preserves almost-representability, so $\Pic^+_{Y_1/K}$ is almost-representable. Therefore, combining Propositions \ref{bddcocycimprep}, \ref{descanL/K}, and \ref{locbddesdata}, we obtain that $\Pic^+_{X/K}$ is almost-representable.
\end{proof}

Next we will prove the uniqueness of the connected group $G^0$ arising in a presentation $\Pic^+_{X/K} \simeq G/E$ of $\Pic^+_{X/K}$ as an almost-representable functor. We require a couple of lemmas.

\begin{lemma}
\label{ExtbyZ}
Let $G$ be a normal connected commutative $K$-group scheme of finite type. Then $\Ext^1(G, \Z^n) = 0$.
\end{lemma}

\begin{proof}
We may assume that $n = 1$. Because $G$ is normal, ${\rm{H}}^1(G, \Z) = 0$. By \cite[Ch.\,II, \S3, Prop.\,2.3]{dg} or \cite[Prop.\,B.2.5]{conradreductive}, therefore, the commutative extensions of $G$ by $\Z$ may be canonically identified with a subgroup of the second Hochschild cohomology group ${\rm{H}}^2_0(G, \Z)$. Because $G$ is geometrically connected, $G \times G$ is connected, so any $2$-cocycle $G \times G \rightarrow \Z$ is constant, hence is a coboundary. It follows that $\Ext^1(G, \Z) = 0$.
\end{proof}

\begin{lemma}
\label{homtoG/E}
Let $G$ be a smooth connected $K$-group scheme, $\mathscr{F}$ a sheaf of abelian groups on the smooth-\'etale site of ${\rm{Spec}}(K)$, $E \subset \mathscr{F}$ a subsheaf represented by an \'etale $K$-group scheme such that $E(K_s)$ is free of finite rank, and $\phi\colon G \rightarrow \mathscr{F}/E$ a homomorphism. Then $\phi$ lifts uniquely to a homomorphism $\psi\colon G \rightarrow \mathscr{F}$.
\end{lemma}

\begin{proof}
We may replace $G$ by $G/\mathscr{D}G$ and thereby assume that $G$ is commutative. Thanks to Galois descent and the uniqueness aspect of the lemma, we may assume that $E \simeq \Z^n$. The obstruction to lifting $\phi$ lies in $\Ext^1(G, E)$, which vanishes by Lemma \ref{ExtbyZ}. Given two lifts $\psi$, $\psi'$, their difference yields a homomorphism $G \rightarrow E$, which must be trivial because $G$ is connected and $E$ is \'etale.
\end{proof}

We may now prove the uniqueness of $G^0$ in Theorem \ref{almostreprestpic}.

\begin{proposition}
\label{G^0unique}
Let $G_1, G_2$ be smooth $K$-group schemes, $E_i \subset G_i$ subsheaves on the smooth-\'etale site represented by \'etale $K$-group schemes such that $E_i(K_s)$ is free of finite rank. Given an isomorphism $\phi\colon G_1/E_1 \xrightarrow{\sim} G_2/E_2$, there is a unique map $\psi\colon G_1^0 \rightarrow G_2^0$ such that the induced map $G_1^0/(E_1 \cap G_1^0) \rightarrow G_2^0/(E_2\cap G_2^0)$ is the restriction of $\phi$. Further, $\psi$ is an isomorphism.
\end{proposition}

\begin{proof}
Consider the map $G_1^0 \rightarrow G_1/E_1 \xrightarrow{\phi} G_2/E_2$. By Lemma \ref{homtoG/E}, this lifts uniquely to a map $G_1^0 \rightarrow G_2$, which must land inside $G_2^0$. This provides our unique $\psi$. To see that it is an isomorphism, apply the existence result to $\phi^{-1}$ to obtain a map $\gamma\colon G_2^0 \rightarrow G_1^0$. Then $\gamma \circ \psi$ provides a lift on the connected component of the identity map of $G_1/E_1$, so by uniqueness it is the identity. A similar argument shows the same for $\psi\circ \gamma$.
\end{proof}

We now restate and prove the main Theorem \ref{almostreprestpic}.

\begin{theorem}$($Theorem $\ref{almostreprestpic}$$)$
\label{almostreprestpicbody}
If $K$ is a field and $X$ is a regular $K$-scheme of finite type, then we have an isomorphism $\Pic^+_{X/K} \simeq G/E$ of sheaves on the smooth-\'etale site of ${\rm{Spec}}(K)$, where $G$ is a smooth commutative $K$-group scheme with finitely generated component group and $E \subset G^t$ is an \'etale $K$-group scheme such that $E(K_s)$ is a free abelian group of finite rank. Furthermore, $G^0$ is unique up to unique isomorphism.
\end{theorem}

\begin{proof}
Combine Theorem \ref{represtpicmainbody}, Lemma \ref{quotalmrep}, and Proposition \ref{G^0unique}.
\end{proof}

\section{Structure of the restricted Picard functor}
\label{structuresection}

In this section we investigate the structure of the restricted Picard functor, which -- thanks to Theorem \ref{represtpicmainbody} -- we now know to be almost-representable for regular, finite type $X/K$. In order to do this, we first require a means of determining when an element of $\Pic^+_{X/K}(S)$ arises from an element of $\Pic(X \times_K S)$, which is the purpose of the following proposition.

\begin{proposition}
\label{lbobstacle}
Let $X$ be a $K$-scheme, and $S$ a smooth $K$-scheme. Then there is an exact sequence, functorial in both $S$ and $X$:
\[
0 \longrightarrow {\rm{H}}^1(S, f_*(\Gm)) \longrightarrow \Pic(X \times_K S) \longrightarrow \Pic^+_{X/K}(S) \longrightarrow {\rm{H}}^2(S, f_*(\Gm)),
\]
where $f\colon X \times S \rightarrow S$ is the projection, and the cohomology is \'etale. If $X(S) \neq \emptyset$, then the above sequence induces an exact sequence
\[
\Pic(X \times_K S) \longrightarrow \Pic^+_{X/K}(S) \longrightarrow {\rm{H}}^2(S, f_*(\Gm)/\Gm).
\]
\end{proposition}

\begin{proof}
All terms in this proof are computed on the smooth-\'etale site of ${\rm{Spec}}(K)$. We first claim that $\Pic^+_{X/K} = \R^1f_*(\Gm)$. Indeed, both sides are the sheafification of the presheaf $T \mapsto \Pic(X \times T)$. The first assertion of the proposition now follows from the Leray spectral sequence (which is functorial in both arguments)
\[
E_2^{i,j} = {\rm{H}}^i(S, \R^jf_*(\Gm)) \Longrightarrow {\rm{H}}^{i+j}(X \times S, \Gm).
\]
For the second part, it is enough thanks to the spectral sequence to show that the map ${\rm{H}}^2(S, \Gm) \rightarrow {\rm{H}}^2(X \times S, \Gm)$ is injective, and this follows from the assumed existence of a section to $f$.
\end{proof}

\begin{corollary}
\label{smconnobst}
If $X$ is a normal $K$-scheme of finite type such that each component admits a $K$-point, then there is a commutative \'etale $K$-group scheme $E$ such that $E(K_s)$ is free of finite rank, and such that we have, for each smooth $K$-scheme $S$ an exact sequence, functorial in $X$ and $S$,
\[
\Pic(X \times_K S) \longrightarrow \Pic^+_{X/K}(S) \longrightarrow {\rm{H}}^2(S, E).
\]
\end{corollary}

\begin{proof}
We may assume that $X$ is connected. Let $f\colon X \rightarrow {\rm{Spec}}(K)$ be the structure map. By Proposition \ref{repofunits}, $f_*(\Gm)$ -- as a sheaf on the smooth-\'etale site of $K$ -- lives in an exact sequence
\[
0 \longrightarrow \R_{A/K}(\Gm) \longrightarrow f_*(\Gm) \longrightarrow E \longrightarrow 0,
\]
where $E$ is as in the corollary and $A \subset \Gamma(X, \calO_X)$ is the (finite) $K$-subalgebra of elements algebraic over $K$. By Proposition \ref{lbobstacle}, therefore, we only need to show that $A = K$. Because $X$ is reduced and connected, $A$ is a field, and because $X(K) \neq \emptyset$, $A = K$.
\end{proof}

\begin{lemma}
\label{H^i(Z, Q/Z)}
We have functorial isomorphisms, for every normal Noetherian scheme $S$, and every $i > 0$, ${\rm{H}}^i(S, \Q/\Z) \simeq {\rm{H}}^{i+1}(S, \Z)$.
\end{lemma}

\begin{proof}
This is a consequence of \cite[Lem.\,5.2.9]{rostateduality} and the exact sequence
\[
\pushQED{\qed} 
0 \longrightarrow \Z \longrightarrow \Q \longrightarrow \Q/\Z \longrightarrow 0. \qedhere
\popQED
\]
\phantom\qedhere
\end{proof}

The following proposition says that a map from a connected group into $\Pic^+_{X/K}$ arises from a line bundle after passing to an isogenous cover.

\begin{proposition}
\label{isogcover}
Let $K$ be a separably closed field, and let $X$ be a normal $K$-scheme of finite type each component of which admits a $K$-point. Then for any smooth connected $K$-group scheme $G$, and any homomorphism $\phi \colon G \rightarrow \Pic^+_{X/K}$ of sheaves on the smooth-\'etale site of ${\rm{Spec}}(K)$, there is a finite \'etale $K$-homomorphism $\psi\colon H \rightarrow G$ from a smooth connected $K$-group scheme $H$ such that $\phi\circ \psi$ arises from a line bundle on $X \times_K H$.
\end{proposition}

\begin{proof}
Combining Corollary \ref{smconnobst} and Lemma \ref{H^i(Z, Q/Z)}, the obstacle to a map $g\colon S \rightarrow \Pic^+_{X/K}$ from a finite type smooth $S/K$ arising from a line bundle on $X \times_K S$ is an element $\Delta_g \in {\rm{H}}^1(S, \Z^n)$ for some $n \geq 0$ (independent of $S$) which is associated in a functorial manner to $g$. Furthermore, this obstacle is additive in $g$: Given $g_i\colon S \rightarrow \Pic^+_{X/K}$, $i = 1,2$, the obstacle for $g_1+g_2\colon S \rightarrow \Pic^+_{X/K}$ is the sum of the obstacles for the $g_i$. 

We apply the above with $S = G \times G$, $g_i = \phi\circ\pi_i$ with $\pi_i\colon G^2 \rightarrow G$ the projections. The map $g_1+g_2$ coincides with $\phi\circ m$, where $m\colon G^2 \rightarrow G$ is addition. Thus we obtain
\[
m^*\Delta_\phi = \Delta_{\phi\circ m} = \Delta_{g_1+g_2} = \Delta_{g_1} + \Delta_{g_2} = \pi_1^*\Delta_{\phi} + \pi_2^*\Delta_\phi.
\]
That is, 
\begin{equation}
\label{noGainPicpfeqn1}
m^*\Delta_\phi = \pi_1^*\Delta_\phi + \pi_2^*\Delta_\phi.
\end{equation}
Now $\Delta_\phi \in {\rm{H}}^1(G, A) \subset {\rm{H}}^1(G, \Q/\Z)$ for some finite constant subgroup scheme $A \subset \Q/\Z$.

Let $h\colon T \rightarrow G$ be the $A$-torsor corresponding to $\Delta_\phi$. Note that pullback along $h$ kills $\Delta_\phi$, so the map $\phi\circ h\colon T \rightarrow \Pic^+_{X/K}$ arises from a line bundle on $X \times_K T$. Because $K$ is separably closed, ${\rm{H}}^1(K, A) = 0$, so there is a point $e \in T(K)$ lying over $1 \in G(K)$. We claim that equation (\ref{noGainPicpfeqn1}) implies that $T$ admits a $K$-group scheme structure with identity $e$ making $h$ into a homomorphism with kernel $A$. Indeed, this follows from the same argument as in the proof of \cite[Th.\,4.12]{colliot-thelene}. That result treats the case when $A$ is of multiplicative type instead of constant, and uses Rosenlicht's theorem that then any map $G^3 \rightarrow A$ sending identity to identity is a character. In fact, all that the argument really uses is that any map $G^3 \rightarrow S$ is a sum $\sum_{i=1}^3 \gamma_i\circ\pi_i$, with $\pi_i\colon G^3 \rightarrow G$ the projection and some $\gamma_i\colon G \rightarrow A$. This holds in our setting with $G$ connected and $A$ \'etale because in fact any map $G^3 \rightarrow A$ must then be constant.

So $T$ is a group scheme admitting a finite \'etale $K$-homomorphism to $G$. It follows that $T^0 \rightarrow G$ is also finite \'etale. Since $T^0 \rightarrow G$ naturally factors through the inclusion $T^0 \rightarrow T$, the composition $T^0 \rightarrow G \rightarrow \Pic^+_{X/K}$ is induced by a line bundle on $X \times_K T^0$.
\end{proof}

The following result says that, quite generally, $\Pic^+_{X/K}$ does not admit nontrivial maps from tori.

\begin{proposition}
\label{noG_minpic}
Let $X$ be a normal $K$-scheme of finite type such that each component of $X_{K_s}$ admits a $K_s$-point. Then there is no nonzero homomorphism $T \rightarrow \Pic^+_{X/K}$ with $T$ a $K$-torus.
\end{proposition}

\begin{proof}
We may extend scalars to $K_s$ and may therefore assume that $T = \Gm$. Let $f\colon \Gm \rightarrow \Pic^+_{X/K}$ be a homomorphism. By Proposition \ref{isogcover}, there is an isogeny $\psi\colon H \rightarrow \Gm$ with $H$ smooth and connected such that $f\circ \psi$ arises from a line bundle. The $K$-group $H$ must also be isomorphic to $\Gm$, hence we are reduced to the case in which $f$ arises from a line bundle. But pullback along the projection induces an isomorphism $\Pic(X) \xrightarrow{\sim} \Pic(X \times_K \Gm)$ \cite[IV$_4$, Cor.\,21.4.11]{ega}, so $f$ is a constant map, hence $0$.
\end{proof}

\begin{proposition}
\label{noGainPic}
If $K$ is a perfect field, and $X$ is a normal $K$-scheme of finite type, then there is no nonzero homomorphism $\Ga \rightarrow \Pic^+_{X/K}$.
\end{proposition}

\begin{proof}
We may extend scalars and assume that $K = \overline{K}$. Let $f\colon \Ga \rightarrow \Pic^+_{X/K}$ be a homomorphism. By Proposition \ref{isogcover}, there is an isogeny $\psi\colon H \rightarrow \Ga$ with $H$ smooth and connected such that $f\circ\psi$ arises from a line bundle. Because $K$ is algebraically closed, $H \simeq \Ga$. Thus we are reduced to the case in which $f$ arises from a line bundle. But pullback induces an isomorphism $\Pic(X) \xrightarrow{\sim} \Pic(X \times_K \Ga)$ \cite[IV$_4$, Cor.\,21.4.11]{ega}, so $f$ is a constant map, hence $0$.
\end{proof}

\begin{remark}
Proposition \ref{noGainPic} fails in general even for smooth $X$ without the perfectness assumption on $K$. More precisely, for every imperfect field $K$, there is a smooth finite type $K$-scheme $X$ such that $\Pic^+_{X/K}$ is (represented by) a nontrivial split unipotent $K$-group scheme; see Example \ref{splitPic}.
\end{remark}

\begin{corollary}
\label{notorus}
Let $X$ be a normal $K$-scheme of finite type such that each component of $X_{K_s}$ admits a $K_s$-point, and let $\Pic^+_{X/K} \simeq G/E$ be a presentation of $\Pic^+_{X/K}$ as an almost-representable functor. Then $G^0$ is an extension of a smooth unipotent group scheme by an abelian variety. If $K$ is perfect, then $G^0$ is an abelian variety.
\end{corollary}

\begin{proof}
First suppose that $K$ is perfect. By Chevalley's theorem \cite{conradchevalley}, we may write $G^0$ as an extension
\[
0 \longrightarrow L \longrightarrow G^0 \longrightarrow A \longrightarrow 0
\]
with $L/K$ a connected linear algebraic group and $A/K$ an abelian variety. We must show that $L = 0$. Because $L$ contains no nontrivial torus by Proposition \ref{noG_minpic}, $L = U$ is unipotent \cite[IV, Cor.\,11.5(2)]{borel}. Because $K$ is perfect, every smooth connected unipotent $K$-group is split \cite[Cor.\,B.2.7, B.3.3]{cgp}, so we are done by Proposition \ref{isogcover}.

Now consider the imperfect case. By \cite[Th.\,A.3.9]{cgp} and Proposition \ref{noG_minpic}, we have an exact sequence
\[
1 \longrightarrow A \longrightarrow G^0 \longrightarrow L \longrightarrow 1
\]
with $A$ an abelian variety and $L$ linear algebraic. We must show that $L$ does not contain a nontrivial torus (and is therefore unipotent). This follows from the fact that $G^0$ does not contain a nontrivial torus and \cite[Prop.\,C.4.5(2)]{cgp}.
\end{proof}

We now prove Theorem \ref{represtpic}.

\begin{theorem}
\label{represtpicbody}$($Theorem $\ref{represtpic}$$)$
Let $K$ be a field and $X$ a regular $K$-scheme of finite type with dense smooth locus. Assume that any map from $X$ into an abelian variety $A$ factors through a finite subscheme of $A$. Then $\Pic^+_{X/K}$ is represented by a smooth commutative $K$-group scheme with finitely generated component group. Furthermore, $(\Pic^+_{X/K})^0$ is  unipotent.
\end{theorem}

\begin{proof}
Let $\Pic^+_{X/K} \simeq G/E$ be an almost-representable presentation of $\Pic^+_{X/K}$ (Theorem \ref{represtpicmainbody}) as in Lemma \ref{quotalmrep}. By Corollary \ref{notorus}, $G^0$ is an extension of a smooth unipotent group scheme by an abelian variety $A$. The key point is to show that $A = 0$. We claim that we may assume that $K$ is separably closed. The only unclear point is that for $X$ not to admit a map with infinite image into an abelian variety is preserved upon passage to $K_s$. For this, we simply note that, given a finite separable extension $L/K$, and a map $X_L \rightarrow B$ (with $B$ an abelian variety over $L$) with infinite image, the corresponding map $X \rightarrow \R_{L/K}(B)$ also has infinite image, where $\R_{L/K}$ denotes Weil restriction. Since $\R_{L/K}(B)$ is also an abelian variety, we may indeed assume that $K$ is separably closed.

We may also assume that $X$ is nonempty and connected, hence irreducible. We will show that there is no nonzero homomorphism $B \rightarrow G^0$ with $B$ an abelian variety, hence $A = 0$. Thanks to Proposition \ref{isogcover}, we may by passing to an isogenous cover of $B$ assume that the homomorphism $B \rightarrow G/E \simeq \Pic^+_{X/K}$ comes from a line bundle $\mathscr{L}$ on $X \times_K B$. Choose a point $x_0 \in X(K)$ (which exists because $X$ has nonempty smooth locus and $K$ is separably closed). We are free to modify $\mathscr{L}$ by a line bundle pulled back from $B$ and thereby assume that $\mathscr{L}|_{x_0} \simeq \calO_B$.

Now $\mathscr{L}$ defines a map $X \rightarrow \Pic_{B/K}$ which sends $x_0$ to $0$. Since $X$ is connected, therefore, it defines a map $X \rightarrow \widehat{B}$ to the dual abelian variety. By hypothesis, this map has finite image. Because $X$ is integral and $X(K) \neq \emptyset$, it is a constant map to a point $b \in B(K)$. This means that $\mathscr{L} \in \Pic(X) \oplus \Pic(B)$, so the map $B \rightarrow \Pic^+_{X/K}$ is constant, hence $0$. This in turn means that the map $h\colon B \rightarrow G$ factors through $E \subset G$. Because $B$ is connected and $E$ \'etale, it follows that $h = 0$, so we conclude that $A = 0$. By Corollary \ref{notorus}, it follows that $G^0$ is smooth and unipotent, and even trivial if $K$ is perfect. Thus $G^t$ is torsion. It follows that $E \subset G^t$ is trivial. Thus $\Pic^+_{X/K} \simeq G$ is a smooth $K$-group scheme with unipotent identity component.
\end{proof}

The situation is simplified when $X$ has no nonconstant units.

\begin{proposition}
\label{PicisnaiveforcertainX}
If $X$ is a normal $K$-scheme of finite type such that $X(K) \neq \emptyset$ and $X_{K_s}$ has no non-constant units, then for every smooth $K$-scheme $S$, $\Pic^+_{X/K}(S) = \Pic(X \times S)/\Pic(S)$.
\end{proposition}

\begin{proof}
The assumptions and Proposition \ref{repofunits} imply that $f_*(\Gm) = \Gm$. The proposition therefore follows from Proposition \ref{lbobstacle}.
\end{proof}

\begin{proposition}
\label{stronglywd}
If $X$ is a regular $K$-scheme of finite type with dense smooth locus satisfying the property that each connected component of $X_{K_s}$ admits no non-constant units, then there is no non-constant map from a smooth unirational $K$-scheme to $\Pic^+_{X/K}$. In particular, there is no such map to $G$, where $G/E$ is a presentation of $\Pic^+_{X/K}$ as an almost-representable functor.
\end{proposition}

\begin{proof}
We may assume that $K = K_s$ and that $X$ is nonempty and connected. Because $X$ has dense smooth locus, $X(K) \neq \emptyset$, so Proposition \ref{PicisnaiveforcertainX} implies that $\Pic(X \times S) \twoheadrightarrow \Pic^+_{X/K}(S)$ is surjective for every smooth $S/K$. We need to show that there are no nonconstant maps to $\Pic^+_{X/K}$ for $S$ unirational. We may assume that $S$ is open in $\P^1$. Then every element of $\Pic^+_{X/K}(S)$ comes from $\Pic(X \times S) = \Pic(X)$ \cite[IV$_4$Cor.\,21.4.11]{ega}. Thus, any map $S \rightarrow \Pic^+_{X/K}$ is constant.
\end{proof}

Finally, we treat $\Pic^+_{X/K}$ when $X$ is a form of affine space (Theorem \ref{formsofaffrep}).

\begin{theorem}
\label{formsofaffrepbody}$($Theorem $\ref{formsofaffrep}$$)$
If $K$ is a field and $X$ is a $K$-form of $\A^n_K$, then $\Pic^+_{X/K}$ is represented by a smooth commutative unipotent $K$-group scheme with strongly wound identity component. If $X(K) \neq \emptyset$, then, for every smooth $K$-scheme $S$, $\Pic^+_{X/K}(S) = \Pic(X \times S)/\Pic(S)$.
\end{theorem}

\begin{proof}
The second statement of the theorem follows from Proposition \ref{PicisnaiveforcertainX}, because $\A^n$ has no non-constant units over a field. Because $\A^n$ has no non-constant map to an abelian variety, Theorem \ref{represtpicbody} implies that $\Pic^+_{X/K}$ is represented by a smooth $K$-group scheme with finitely-generated component group. Proposition \ref{stronglywd} implies that $(\Pic^+_{X/K})^0$ is strongly wound, so it only remains to show that the component group of $\Pic^+_{X/K}$ is $p$-power torsion, for which it suffices to show the same for $\Pic^+_{X/K}(K_s) = \Pic(X_{K_s})$. For this we may assume that $K = K_s$. We have that $\Pic(X_L) = \Pic(\A^n_L) = 0$ for some finite extension $L/K$, so the assertion follows from \cite[Lem.\,3]{rosaffinenesscrit}.
\end{proof}

\begin{example}
\label{picnotnaive}
The conclusion that $\Pic^+_{X/K}(S) = \Pic(X \times S)/\Pic(S)$ in Theorem \ref{formsofaffrepbody} really does require the assumption that $X(K) \neq \emptyset$ in general, even in the case $S = {\rm{Spec}}(K)$, as the following example shows. (This also provides an example of a point of $\Pic^+_{X/K}(K)$ not arising from a line bundle without passing to an \'etale cover.) Denote by $f\colon X \rightarrow {\rm{Spec}}(K)$ the structure map. The spectral sequence
\[
E_2^{i,j} = {\rm{H}}^i(K, \R^jf_*(\Gm)) \Longrightarrow {\rm{H}}^{i+j}(X, \Gm)
\]
shows that the map $\Pic(X) \rightarrow \Pic^+_{X/K}(K)$ is surjective if and only if the pullback map ${\rm{H}}^2(K, \Gm) \rightarrow {\rm{H}}^2(X, \Gm)$ is injective. We will construct for any field $K$ of characteristic $p$ with $\Br(K)[p] \neq 0$ a $K$-form $X$ of $\A^1$ for which this injectivity fails.

By \cite[Prop.\,2.6.4]{rostateduality}, for any regular ring $A$ of characteristic $p$, there is a functorial (in $A$) surjective group homomorphism
\begin{equation}
\label{picnotnaiveeqn1}
\phi_A\colon \frac{\Omega^1_A}{dA + (C^{-1}-i)(\Omega^1_A)} \twoheadrightarrow {\rm{H}}^2(A, \Gm)[p],
\end{equation}
where $d\colon A \rightarrow \Omega^1_A$ is the universal derivation, $C^{-1}\colon \Omega^1_A \rightarrow \Omega^1/dA$ is the inverse Cartier map defined by $C^{-1}(fdg) = f^pg^{p-1}dg$ (and extended by linearity) and $i\colon \Omega^1_A \rightarrow \Omega^1_A/dA$ is the projection. By assumption, ${\rm{H}}^2(K, \Gm)[p] \neq 0$, so it follows that there exist $\lambda, t \in K$ such that $\phi_k(\lambda dt) \neq 0 \in \Br(K)$. We will construct a $K$-form $X$ of $\A^1$ such that $\phi_K(\lambda dt)$ pulls back to $0 \in {\rm{H}}^2(X, \Gm)$.

We define
\[
X := \{Y^p + \lambda = t^{p-1}Z^p - Z\} \subset \A^2_K.
\]
First we check that $X$ is indeed a form of $\A^1$. Over $K(t^{1/p})$, one has the inverse isomorphisms
\[
\begin{tikzcd}
X_{K(t^{1/p})} \arrow[r, shift left] & \A^1_{K(t^{1/p})} \arrow[l, shift left] \\
(Y, Z) \arrow[r, mapsto] & Y - t^{(p-1)/p}Z \\
(S-t^{(p-1)/p}S^p-t^{(p-1)/p}\lambda, -S^p-\lambda) & S \arrow[l, mapsto]
\end{tikzcd}
\]
Now we check that $\phi_K(\lambda dt) \mapsto 0 \in {\rm{H}}^2(X, \Gm)$. Indeed, if $B$ denotes the coordinate ring of $X$, then we have an equality in $\Omega^1_B/dB$,
\[
(C^{-1}-i)(Zdt) = (Z^pt^{p-1}-Z)dt = (Y^p+\lambda)dt = d(Y^pt) + \lambda dt.
\]
so $\lambda dt \in (C^{-1}-i)(\Omega^1_B)$ as an element of $\Omega^1_B/dB$. By (\ref{picnotnaiveeqn1}), therefore, we see that $f^*(\phi_K(\lambda dt)) = \phi_B(\lambda dt) = 0$.
\end{example}

\begin{example}
\label{splitPic}
Let $K$ be an imperfect field, and let $U$ be a smooth connected unipotent $K$-group with $\Pic^+_{U/K}$ of strictly positive dimension. (Equivalently, since $\Pic^+_{U/K}$ is of finite type by Theorem \ref{formsofaffrepbody}, $\Pic(U_{K_s})$ is infinite.) Such examples exist over every imperfect field; see \cite[Th.\,6.7.9]{kmt} and \cite[Props.\,5.2, 5.4, 5.5, Ex.\,5.8]{rostrans}. There is a finite \'etale closed $K$-subgroup scheme $E \subset \Pic^+_{U/K}$ such that $\Pic^+_{U/K}/E$ is (nontrivial) split unipotent. Indeed, \cite[Lem.\,2.1.5]{rostateduality} implies this for the identity component of $\Pic^+_{U/K}$, and we can then further quotient by a finite \'etale $K$-subgroup of the torsion group $\Pic^+_{U/K}$ that surjects onto its component group.

By Theorem \ref{formsofaffrepbody}, the inclusion $E \subset \Pic^+_{U/K}$ comes from a line bundle $\mathscr{L} \in \Pic(U \times E)$. Let $F/K$ be a finite Galois extension splitting $E$. For each point $e \in E(F)$, $\mathscr{L}_e := \mathscr{L}|U \times \{e\}$ is represented by an effective divisor $D_e$ on $U$. Let $D$ be the sum of the Galois conjugates of $\sum_{e \in E(F)} D_e$. Then $D$ is a divisor on $U$ such that $\mathscr{L}|(U-D) \times E$ is trivial. Therefore, the composition $E \subset \Pic^+_{U/K} \rightarrow \Pic^+_{(U-D)/K}$ vanishes. We claim that the map $\Pic^+_{U/K} \rightarrow \Pic^+_{(U-D)/K}$ is surjective with kernel of strictly smaller dimension than ${\rm{dim}}(\Pic^+_{U/K})$ (in fact, the kernel is finite, but we will not require this), from which it will follow that $\Pic^+_{(U-D)/K}$ is a nontrivial quotient of the {\em split} unipotent $K$-group $\Pic^+_{U/K}/E$. We will have therefore constructed over every imperfect field $K$ a smooth connected $K$-scheme whose restricted Picard scheme is nontrivial split unipotent.

To prove the desired surjectivity, it suffices because $\Pic^+_{(U-D)/K}$ is smooth to prove surjectivity on $K_s$-points. Because all $K_s$-points on both schemes come from line bundles, this follows from the surjectivity of $\Pic(U_{K_s}) \rightarrow \Pic((U-D)_{K_s})$, which holds because both schemes are regular and the latter is open in the former. As for the claim about the kernel, if the kernel had dimension equal to that of $\Pic^+_{U/K}$, then it would contain all of $(\Pic^+_{U/K})^0$ and therefore have infinitely many $K_s$-points. But the kernel of $\Pic(U_{K_s}) \rightarrow \Pic((U-D)_{K_s})$ is generated by the line bundles associated to the finitely many closed points of $(U-D)_{K_s}$. Since it is also torsion, because $\Pic(U_{K_s})$ is, it must be finite.
\end{example}

As a consequence of Theorem \ref{formsofaffrepbody}, we will deduce that, over a separably closed field, infinitude of the Picard group of a unipotent group is equivalent to infinitude of the translation-invariant line bundles. For this we require the following lemma.

\begin{lemma}
\label{uniponunip}
Let $U, V$ be smooth unipotent $K$-groups equipped with a $U$-action on $V$ $($as an algebraic group$)$.
\begin{itemize}
\item[(i)] If $V \neq 1$, then there is a nontrivial smooth $K$-subgroup of $V$ on which $U$ acts trivially.
\item[(ii)] If ${\rm{dim}}(V) > 0$, then there is a strictly positive-dimensional smooth $K$-subgroup of $V$ on which $U$ acts trivially.
\end{itemize}
\end{lemma}

\begin{remark}
For our purposes, we could actually make do with (ii) in the case in which $U$ and $V$ are connected, which would simplify the proof below.
\end{remark}

\begin{proof}
Consider the smooth unipotent $K$-group $W := U \ltimes V$. Define a chain of subgroups $V_i$ of $V$ by the formulas
\begin{align*}
& V_0 := V \\
& V_{i+1} := [U, V_i], \hspace{.1 in} i \geq 0.
\end{align*}
One readily verifies inductively that the $V_i$ are normalized by $U$, and that $V_{i+1} \subset V_i$, so in particular, $V_i \subset V_0 = V$. Further, we claim that the $V_i$ are smooth. We proceed by induction, the case $i = 0$ being immediate. The formation of $V_i$ commutes with arbitrary base change, so we may assume that $K = \overline{K}$ to prove smoothness. But $V_{i+1}$ is the image of a map $(U \times V_i)^n \rightarrow W$, hence reduced. Since $K= \overline{K}$, it follows that $V_{i+1}$ is geometrically reduced, hence smooth, as claimed. Now, by definition, $V_i \subset \mathscr{D}_i(W)$, the $i$th step in the descending central series of $W$. Because $W$ is unipotent, it is nilpotent, hence $\mathscr{D}_i(W) = 1$ for large $i$. If we then let $i_0 \geq 0$ be the largest integer such that $V_{i_0} \neq 1$, then $V_{i_0}$ is centralized by $U$, which proves (i).

For (ii), we first note that $V^0$ is normalized by $U$, hence we may assume that $V$ is connected. Now consider the $K$-subgroups $V'_i$ of $V$ defined as follows:
\begin{align*}
& V'_0 := V \\
& V'_{i+1} := [U^0, V'_i], \hspace{.1 in} i \geq 0.
\end{align*}
Then one readily verifies inductively that $V'_i$ is normalized by $U$ and contained in $V'_{i+1}$. It is connected by a similar argument to that given for smoothness of the $V_i$ above. Just as with the $V_i$ we conclude that $V'_i = 1$ for all large enough $i$. If we let $i_0 \geq 0$ be the maximal integer such that $V'_{i_0} \neq 1$, then $V'_{i_0} \subset V$ is smooth, connected, positive-dimensional, normalized by $U$, and acted on trivially by $U^0$. Replacing $V$ by $V'_{i_0}$, therefore, and $U$ by $U/U^0$, we may assume that $U = E$ is finite \'etale.

The $E$-invariants $V^E$ of $V$ form a closed $K$-subgroup scheme of $V$. The assertion we wish to prove is that the maximal smooth $K$-subgroup scheme of $V^E$ is positive-dimensional. Since the formation of both $V^E$ and the maximal smooth subgroup scheme commute with separable base change \cite[Lem.\,C.4.1]{cgp}, we are free to assume that $K$ is separably closed. We then prove the assertion by induction on $\#E$, the case $\#E = 1$ being trivial. So suppose that $E \neq 1$ is an \'etale unipotent $K$-group. Because $E$ is unipotent, it has nontrivial center, so choose $1 \neq e \in E(K)$ lying in the center of $E$.

Consider as before the $K$-group $W := E \ltimes V$, and define the $K$-subgroups $V''_i \subset V$ by the following formula:
\begin{align*}
& V''_0 := V \\
& V''_{i+1} := [e, V''_i], \hspace{.1 in} i \geq 0,
\end{align*}
where $[e, V''_i]$ means the group generated by the map $V''_i \rightarrow W, v \mapsto [e,v]$. As before, one readily verifies that $V''_i$ is a descending series of smooth connected subgroups of $V$, all of which are normalized by $E$ (this uses the fact that $e$ is central in $E$), and again, as before, unipotence of $W$ implies that there is some $i_0$ such that $V''_{i_0}$ is positive-dimensional and $e$ acts trivially on $V''_{i_0}$. Replacing $V$ by $V''_{i_0}$, then, we may suppose that $e$ acts trivially on $V$. Then the action descends to an action of $E/\langle e\rangle$ on $V$, and we are therefore done by induction.
\end{proof}

Recall that, for a smooth connected $K$-group scheme $G$, $$\Ext^1(G, \Gm) := \{\mathscr{L} \in \Pic(G) \mid m^*\mathscr{L} \simeq \pi_1^*\mathscr{L} \otimes \pi_2^*\mathscr{L}\} \subset \Pic(G).$$ That is, $\Ext^1(G, \Gm)$ consists of the universally translation-invariant line bundles on $G$. For an explanation of the notation, see the introduction to \cite{rostrans}.

\begin{proposition}
\label{Picfiniteiffextfinite}
Let $U$ be a smooth connected unipotent $K$-group.
\begin{itemize}
\item[(i)] $\Pic(U_{K_s})$ is trivial if and only if $\Ext^1_{K_s}(U, \Gm)$ is.
\item[(ii)] $\Pic(U_{K_s})$ is finite if and only if $\Ext^1_{K_s}(U, \Gm)$ is.
\end{itemize}
\end{proposition}

\begin{proof}
We may assume that $K$ is separably closed. The only if directions are immediate, so we only need to prove the if directions. By Theorem \ref{formsofaffrepbody}, $\Pic^+_{U/K}$ is represented by a (smooth) unipotent $K$-group scheme. Furthermore, $\Pic^+_{U/K}(K) = \Pic(U)$. Now the action of $U$ on itself by left-translation induces an action on $\Pic^+_{U/K}$, and $x \in \Pic^+_{U/K}(K)$ is fixed by $U$ if and only if the corresponding line bundle lies in $\Ext^1(U, \Gm)$. Thanks to Lemma \ref{uniponunip}, we have that $((\Pic^+_{U/K})^U)^{\rm{sm}}$ (with ${\rm{sm}}$ denoting the maximal smooth $K$-subgroup scheme) is trivial, respectively \'etale, if and only if the same holds for $\Pic^+_{U/K}$, which proves the proposition.
\end{proof}

\noindent \address
\vspace{.3 in}

\noindent \email

\end{document}